\documentclass[a4paper, reqno, 11pt]{amsart}

\usepackage{kotex}
\usepackage{amssymb}
\usepackage{colortbl}
\usepackage[dvipsnames]{xcolor}
\usepackage{amsmath}
\usepackage{amsrefs}
\usepackage{a4wide}
\usepackage{enumitem}
\usepackage[top=1.in, bottom=1.in, left=1.in, right=1.in]{geometry}

\newtheorem{theorem}{Theorem}[section]
\newtheorem{lemma}[theorem]{Lemma}

\newtheorem{definition}[theorem]{Definition}
\newtheorem{proposition}[theorem]{Proposition}

\newtheorem{remark}[theorem]{Remark}

\numberwithin{equation}{section}

\makeatletter
\@namedef{subjclassname@2020}{\textup{2020} Mathematics Subject Classification}
\makeatother

\newcommand{\mr}{\mathbb{R}}

\newcommand{\rom}[1]{\uppercase\expandafter{\romannumeral #1\relax}}
\newcommand{\ds}{\displaystyle}
\newcommand{\ep}{\epsilon}

\newcommand{ \R }{ \mathbb{R} }

\def\Xint#1{\mathchoice
    {\XXint\displaystyle\textstyle{#1}}%
     {\XXint\textstyle\scriptstyle{#1}}%
     {\XXint\scriptstyle\scriptscriptstyle{#1}}%
     {\XXint\scriptstyle\scriptscriptstyle{#1}}%
    \!\int}
\def\XXint#1#2#3{{\setbox0=\hbox{$#1{#2#3}{\int}$}
    \vcenter{\hbox{$#2#3$}}\kern-.5\wd0}}

\begin{document}

\today

\title[]{Local H\"older continuity for fractional nonlocal equations with general growth}

\author[]{Sun-Sig Byun, Hyojin Kim \and Jihoon Ok}

\address{Sun-Sig Byun: Department of Mathematical Sciences and Research Institute of Mathematics, Seoul National University, Seoul 08826, Korea}
\email{byun@snu.ac.kr}

\address{Hyojin Kim: Department of Mathematical Sciences, Seoul National University, Seoul 08826, Korea}
\email{hyojin@snu.ac.kr}

\address{Jihoon Ok: Department of Mathematics, Sogang University, Seoul 04107, Korea}
\email{jihoonok@sogang.ac.kr}

\keywords{Nonlocal operator; general growth; N-function; local boundedness; H\"older continuity}

\subjclass[2020]{
35R11; %Fractional partial differential equations
47G20; %Integro-differential operators
35D30;  % Weak solutions to PDEs
35B65. % Smoothness and regularity of solutions to PDEs
}

\thanks{S. Byun was supported by NRF-2021R1A4A1027378. H. Kim was supported by 2020R1C1C1A01009760, J. Ok was supported by NRF-2017R1C1B2010328.}

\begin{abstract}
We study generalized fractional $p$-Laplacian equations to prove local boundedness and H\"older continuity of weak solutions to such nonlocal problems by finding a suitable fractional Sobolev-Poincar\'e inquality.

\end{abstract}

\maketitle

\section{Introduction}\label{sec1}

In this paper we study the fractional nonlocal equation
\begin{align}\label{pde}
\mathcal Lu=0  \quad \mbox{in } \ \Omega
\end{align}
defined on a bounded domain $\Omega$ in $\mathbb{R}^n$
with $n \geq 2$ by
\begin{align}\label{op}
\mathcal
Lv(x):=\mathrm{p.v.}\int_{\mr^n}g\left(\frac{|v(x)-v(y)|}{|x-y|^s}\right)\frac{v(x)-v(y)}{|v(x)-v(y)|}K(x,y)\frac{dy}{|x-y|^{s}},
\end{align}
where $0<s<1$, $g:[0,\infty)\to [0,\infty)$ is a strictly
increasing, continuous function that satisfies $g(0)=0$, $\ds
\lim_{t\to\infty}g(t)=\infty$ and
\begin{align}\label{growth}
1< p \le \frac{tg(t)}{G(t)} \le q < \infty \quad \text{for some }\
1<p\le q, \quad \mbox{where }\ G(t):=\int_0^tg(s)\,ds.
\end{align}
$K:\mr^{n} \times \mr^{n} \rightarrow (0,\infty]$ is a symmetric,
i.e., $K(x,y)=K(y,x)$, and measurable kernel that satisfies
\begin{align}\label{kernel}
\frac{\lambda}{|x-y|^{n}} \le K(x,y) \le \frac{\Lambda}{|x-y|^{n}},
\qquad x,y\in \R^n,
\end{align}
for some $0<\lambda \le \Lambda$. A main point %of this paper 
is that
the function $G$ is an N-function satisfying the $\Delta_{2}$ and
$\nabla_{2}$ conditions (see the next section) and that a simple
example of the kernel $K(x,y)$ is $a(x,y)|x-y|^{-n}$ with $\lambda
\le a \le \Lambda$. Note in particular case when
$K(x,y)=|x-y|^{-n}$, $\mathcal L$ is the so-called $s$-fractional
$G$-Laplace operator and we denote it by $\mathcal L=(-\Delta)_{G}^{s}$.

The goal of this paper is to establish the local H\"older regularity for
the nonlocal problem \eqref{pde} without a priori boundedness assumption of
weak solutions. In addition we will discuss the existence and
uniqueness of weak solution to \eqref{pde} with a Dirichlet
boundary condition.

An obvious example is $g(t)=t$ and $K(x,y)=|x-y|^{-n}$, in which
case it reduces to the $s$-fraction Laplace operator $(-\Delta)^s$
and Caffarelli, Chan and Vasseur \cite{CCV1} proved H\"older
regularity result for the corresponding parabolic fractional
equation by the approach of De Giorgi modified to the nonlocal
setting. We refer to
\cite{CCV1,CS,Kas1,Kas2,KMS2,Sil,MSY} for various regularity results, including the
Harnack inequality, self improving property and $L^p$-regularity,
for weak solutions to the fractional nonlocal linear equations. On
the other hand for the fractional $p$-Laplacian type equations,
i.e., $g(t)=t^{p-1}$ with $1<p<\infty$, Di Castro, Kuusi and
Palatucci  \cite{DKP2} proved local H\"older regularity by employing the so-called tail(see
the next section). We further refer to
\cite{Coz,DKP1,DZZ,GK,GKS,KKL,KKP1,KKP2,KMS1,Lin,Now1,Now2} for
studies on the nonlocal nonlinear equations of the fractional
$p$-Laplacian type.

A general non-autonomous fractional nonlocal operator can be written as
\begin{align*}
\mathcal
Lv(x):=\mathrm{p.v.}\int_{\mr^n}h\left(x,y,\frac{|v(x)-v(y)|}{|x-y|^s}\right)\frac{v(x)-v(y)}{|v(x)-v(y)|}K(x,y)\frac{dy}{|x-y|^{s}}.
\end{align*}
If $h(x,y,t) \approx t^{p-1}$, then we say that the operator or
equation satisfies the $p$-growth condition. On the other hand, if
$g(x,y,t)$ has a more general structure, then we say that the
operator or equation satisfies a non-standard growth condition.
Typical examples of non-standard growth conditions include the
variable growth condition: $h(x,y,t)\approx t^{p(x,y)-1}$, the
double phase condition: $h(x,y,t)\approx t^{p-1}+a(x,y)t^{q-1}$, and
the general growth condition: $h(x,y,t)\approx g(t)$. Recently
there has been a great deal of studies concerning fractional
nonlocal equations with nonstandard growth conditions, in particular
for H\"older regularity in \cite{Ok1,CK1} with the variable growth
condition and in \cite{BOS,DP,FZ} with the double phase condition,
respectively.

We are mainly focusing on the general growth condition. The local
one corresponding to the nonlocal equation \eqref{pde} is the so
called $G$-Laplace equation:
\begin{align*}
\mathrm{div}\, \left(g(|Du|)\frac{Du}{|Du|}\right)=0\quad \mbox{in}\
\ \Omega, \quad \text{where }\ g(t)=G'(t),
\end{align*}
for which Lieberman \cite{Lie1} proved the $C^{0,\alpha}$ and
$C^{1,\alpha}$ continuity of weak solutions under sharp conditions
on $g$ like \eqref{growth}. We also refer to
\cite{DSV1,HHL1,HO} and references therein for the regularity
results regarding generalized $p$-Laplace equations. It is interesting
to consider the assertion $C^{0,\alpha}$ continuity as a natural
outgrowth to the nonlocal equation \eqref{pde}. Now, we can see such
H\"older regularity results in the recent papers \cite{FSV1, FSV2,
CKW1}. However, they are established with boundedness assumption of
weak solutions or under rather restrictive conditions on $g$. Our
intention in this paper is to prove the local H\"older regularity of
weak solutions to \eqref{pde} without this a priori assumption and
other condition than \ref{growth}.

With the definition of weak solution, the related function spaces
and the tail to be introduced in details in the next section,
we now state our main result.

\begin{theorem}\label{mainthm}
Let $0<s<1$. Suppose that $u \in \mathbb{W} ^{s,G}(\Omega) \cap L_{s}^{g}(\mr^{n})$
is a weak solution to \eqref{pde} with  \eqref{op}, \eqref{growth} and
\eqref{kernel}. Then $u \in C_{\mathrm{loc}}^{0,\alpha}(\Omega)$ for some
$\alpha \equiv \alpha(n,s,p,q,\lambda,\Lambda)\in(0,1)$. Moreover, there exist
positive constants $c_b$ and $c_h$ depending on $n,s,p,q,\lambda$
and $\Lambda$ such that for any $B_{r}(x_{0}) \Subset  \Omega$,
\begin{align}
\label{lb1} \|u\|_{L^\infty(B_{r/2}(x_0))} \le c_b r^{s}G^{-1}\left(
\Xint-_{B_r(x_0)}G\left( \frac{|u|}{r^s} \right)dx \right) +
r^{s}g^{-1}(r^s\mathrm{Tail}(u;x_0,r/2))
\end{align}
and
\begin{align}\label{holder}
[ u ]_{C^{0,\alpha}(B_{r/2}(x_{0}))} \le
\frac{c_h}{r^{\alpha}}\left[ r^{s}G^{-1} \left(
\Xint-_{B_{r}(x_0)}G\left( \frac{|u|}{r^{s}} \right)dx \right) +
r^{s}g^{-1}(r^{s}\mathrm{Tail}(u;x_{0},r))\right].
\end{align}
\end{theorem}

\begin{remark}
We can consider a weaker condition on $g$ that $g(0)=0$, $\ds
\lim_{t\to\infty}g(t)=\infty$, and
\begin{align*}
\frac{g(s)}{s^{p-1}} \le L \frac{g(t)}{t^{p-1}} \ \ \mbox{and}\ \
\frac{g(t)}{t^{q-1}} \le L \frac{g(s)}{s^{q-1}} \quad \mbox{for all
}\ 0<s\le t,
\end{align*}
for some $1<p<q$ and $L\ge 1$. In fact, we see that there exists a
function $\tilde g$ such that $\tilde g \approx g$ and $\tilde g$
satisfies the assumption of $g$ in \eqref{growth}. See \cite[Chapter
2]{HH} for details.
\end{remark}

Our proof of the theorem is based on the De Giorgi approach established in \cite{DKP2}, in particular, for the fractional $p$-Laplace type equations in the setting of fractional Sobolev space $W^{s,p}$. On the other hand, to the fractional $G$-Laplace type equations,  this approach can not be directly applied, as $G(st)\not\approx G(s)G(t)$ (this equivalence is true when $G(t)=t^p$). Indeed, we are forced to face a more complicated and delicate situation under which we need to make a very careful systematic analysis to overcome the complexity and difficulty coming from such a $G$-Laplace type nonlocal problem. 
Moreover, an integral version of Sobolev-Poincare type inequality plays an
essential role in the process of De Giorgi iteration, which is not
known in the fractional Orlicz-Sobolev space as of today, as far as
we are concerned. Therefore in this paper we obtain this estimate in
Lemma~\ref{lem.sobopoin}.

The paper is organized as follows. In the next section we introduce
notations, functions spaces. weak solutions and fundamental
inequalities that will be used throughout this paper.
Section~\ref{sec3} is devoted to deriving two essential estimates
for weak solutions to \eqref{pde}. One is a Caccioppoli type
inequality and the other is a logarithmic estimate. In
Section~\ref{sec4} we prove Theorem~\ref{mainthm} by proving the
local boundedness and then H\"older continuity of weak solutions.
The final section includes the existence and uniqueness of weak
solutions to \eqref{pde} with the Dirichlet boundary condition.

\section{Preliminaries}\label{sec2}

In this paper $B_r(x_0)$ denotes the ball in $\R^n$ with radius $r>0$ centered at $x_0 \in \mathbb{R}^n$. When the center is clear in the context, we write it
by $B_r$ for the sake of simplicity. The average of an integrable
function $f$ on $B_r$ is defined as
\begin{align*}
(f)_{B_r}=\Xint-_{B_r} f\, dx =\frac{1}{|B_r|} \int_{B_r} f\,dx.
\end{align*}
Throughout this paper we denote by $c$ to mean a universal constant that can be computed by given quantities 
such as $n, s, p, q, \lambda, \Lambda$. This generic constant  can vary from line to line.

\subsection{N-functions}

A measurable function $G:[0,\infty) \rightarrow [0,\infty)$ is
called an \textit{N-function}(nice Young function) if it is a Young
function, i.e., it is increasing and convex and $G(0)=0$, and
satisfies
\begin{align*}
%\label{Nftn} 
\lim_{t \rightarrow
0+}\frac{G(t)}{t} = 0 \quad \mbox{and} \quad  \lim_{t \rightarrow
\infty}\frac{G(t)}{t} = \infty.
\end{align*}
For an N-function $G:[0,\infty) \rightarrow [0,\infty)$, its
conjugate function $G^{*}:[0,\infty)\to [0,\infty)$ of $G$ is
defined by
\begin{align*}
G^{*}(t):=\sup_{s \ge 0}\, (st-G(s)), \quad t\ge0.
\end{align*}
Throughout this paper we always assume that $G\in C^1([0,\infty))$
satisfies \eqref{growth}. Note from \eqref{growth} that for
$t\in[0,\infty)$,
\begin{align}\label{Gineq}
a^q G(t)  \le G(at) \le a^p G(t)\  \text{if}  \ 0<a<1 \  \mbox{and}\  a^p G(t)  \le G(at) \le a^q G(t) \ \text{if}\   a>1,
\end{align}
and that
\begin{align}\label{G*ineq}
a^{p'} G^*(t)  \le G^*(at) \le a^{q'} G^*(t)\  \text{if}  \ 0<a<1 \  \mbox{and}\  a^{q'} G^*(t)  \le G^*(at) \le a^{p'} G^*(t) \  \text{if}\  a>1.
\end{align}
where $p'$ and $q'$ are the H\"older conjugates of $p$ and $q$, respectively.
Also we see that $G$ satisfies the following $\Delta_2$ and
$\nabla_2$ conditions (see \cite[Proposition 2.3]{MR}):
\begin{itemize}
\item[($\Delta_{2}$)] there exists a constant $\kappa > 1$ such that
\begin{align}\label{del2}
G(2t) \le \kappa G(t) \quad \mbox{for all} \quad t \ge 0;
\end{align}
\item[($\nabla_{2}$)]
there exists a constant $l > 1$ such that
\begin{align}\label{na2}
G(t) \le \frac{1}{2l}G(lt) \quad \mbox{for all} \quad t \ge 0,
\end{align}
\end{itemize}
where the constants $\kappa$ and $l$ are to be determined by $q$ and $p$.
Note that $G$ satisfies the  $\nabla_2$ condition if and only if $G^*$ does the $\Delta_2$ condition.
In addition, from the definition of the conjugate function, we have
\begin{align}\label{Young0}
ts \le G(t)+G^*(s), \qquad t,s\ge 0.
\end{align}
From \eqref{Gineq}, we deduce that for every $\epsilon\in(0,1)$
\begin{align}\label{Young}
ts \le \epsilon^{1-q} G(t)+ \epsilon G^*(s), \qquad t,s\ge 0,
\end{align}
which is Young's inequality with $\epsilon$. We further have from
\eqref{growth} that
\begin{align}\label{Young1}
G^{*}(g(t)) = tg(t) - G(t) \le (q-1)G(t), \quad t\ge 0.
\end{align}
Also the convexity and \eqref{Gineq} imply
$$
2^{-1} (G(t) + G(s)) \le G(t+s) \le 2^{q-1} (G(t)+G(s)),
$$
which will be used often later in this paper.

\subsection{Fractional Orlicz-Sobolev spaces}

For an open subset $U$ in $\mathbb{R}^n$,  we denote by $\ds
\mathcal{M}(U)$ to mean the class of all real-valued
measurable functions on $U$. For an N-function $G$ satisfying the $\Delta_2$ and $\nabla_2$ conditions, we
define the \textit{Orlicz space} $L^{G}(U)$
as
\begin{align*}
L^{G}(U):=\left\{ v \in \mathcal{M}(U) \ \Big| \
\int_{U}G(|v(x)|)\,dx<\infty \right\},
\end{align*}
which is a Banach space with the Luxemburg norm defined as
\begin{align*}
\|v\|_{L^{G}(U)}:= \inf \left\{ \lambda>0 \ \Big| \ \int_{U}G\left( \frac{|v(x)|}{\lambda} \right)dx \le 1 \right\}.
\end{align*}
Then note that
\begin{align}\label{lux}
\|v\|_{L^{G}(U)} \le \int_{U}G\left( |v| \right)dx + 1
\end{align}
We next let $0<s<1$ and define the \textit{fractional Orlicz-Sobolev
space} $W^{s,G}(U)$ as
\begin{align*} W^{s,G}(U) := \left\{
v \in L^{G}(U) \ \Big| \ \int_{U}\int_{U} G\left(
\frac{|v(x)-v(y)|}{|x-y|^{s}} \right)\frac{dxdy}{|x-y|^{n}} < \infty
\right\},
\end{align*}
which is also a Banach space with the norm
\begin{align*}
\|v\|_{W^{s,G}(U)}:= \|v\|_{L^{G}(U)} + [v]_{s,G,U},
\end{align*}
where $[v]_{s,G,U}$ is the Gagliardo semi-norm defined by
\begin{align*}
[v]_{s,G,U} := \inf \left\{ \lambda>0 \ \Big| \ \int_{U}
\int_{U}G\left( \frac{|v(x)-v(y)|}{\lambda|x-y|^{s}}
\right)\frac{dxdy}{|x-y|^{n}} \le 1 \right\}.
\end{align*}
Thus we have
\begin{align}\label{gar}
[v]_{s,G,U} \le \int_{U}\int_{U}G\left( \frac{|v(x)-v(y)|}{|x-y|^{s}} \right)\frac{dxdy}{|x-y|^{n}} + 1.
\end{align}
Now we introduce the function space to which weak solutions of
\eqref{op} belong,  see the next subsection for the concept of weak
solution. We write
$C_{\Omega}:=(\Omega \times \mr^{n}) \cup (\mr^{n} \times \Omega)$.
Then the space  $\mathbb{W}^{s,G}(\Omega)$ consists of all functions  $v \in \mathcal{M}(\mr^{n})$ with
$\ v |_{\Omega} \in L^{G}(\Omega)$ and
$$\iint_{C_{\Omega}} G\left( \frac{|v(x)-v(y)|}{|x-y|^{s}} \right)\frac{dxdy}{|x-y|^{n}} < \infty.$$
Note that if $v \in \mathbb{W}^{s,G}(\Omega)$, then $v |_{\Omega} \in W^{s,G}(\Omega)$.

\subsection{Weak solution and tail}
We first recall $g$ with \eqref{growth} and $K$ with \eqref{kernel}
to define a weak solution to \eqref{pde}.
\begin{definition}\label{sol}
 $ u \in \mathbb{W}^{s,G}(\Omega)$ is a
 weak solution (resp. subsolution or supersolution) to \eqref{pde} if
\begin{align*}
\iint_{C_{\Omega}}g\left(\frac{|u(x)-u(y)|}{|x-y|^s}\right)\frac{u(x)-u(y)}{|u(x)-u(y)|}(\eta(x)-\eta(y))K(x,y)\,dxdy = 0 \ \text{(resp. $\le0$ or $\ge 0$)}
\end{align*}
for any $\eta \in \mathbb{W}^{s,G}(\Omega)$ (resp. nonnegative $\eta
\in \mathbb{W}^{s,G}(\Omega)$) such that $\eta=0$ in $\mr^{n}
\setminus \Omega$.
\end{definition}

We next write
$$
L^{g}_{s}(\mr^{n}):= \left\{ u \in \mathcal{M}(\mr^{n}) : \int_{\mr^{n}}g\left( \frac{|u(x)|}{(1+|x|)^{s}} \right) \frac{dx}{(1+|x|)^{n+s}}< \infty \right\},
$$
and the tail of $u \in L^{g}_{s}(\mr^{n})$ for the ball $B_R(x_0)$ is denoted by
\begin{align}
\label{tail}
\mbox{Tail}(u;x_{0},R) := \int_{\mr^{n} \setminus B_{R}(x_{0})}g\left( \frac{|u(x)|}{|x-x_{0}|^s} \right)\frac{dx}{|x-x_{0}|^{n+s}}.
\end{align}
We notice that  $u \in L^{g}_{s}(\mr^{n})$ if and only if $\mbox{Tail}(u;x_{0},R) < \infty$ for all $x_{0} \in \mr^{n}$ and $R > 0$. 
Indeed, for $x \in \mr^{n} \setminus B_{R}(x_0)$, a direct computation leads to
$$ \frac{1+|x|}{|x-x_0|} \le 1+\frac{1+|x_0|}{R}.$$
Then it follows from \eqref{Gineq} that
\begin{align*}
\mbox{Tail}(u;x_{0},R) &= \int_{\mr^{n} \setminus B_{R}(x_{0})}g\left( \frac{|u(x)|}{(1+|x|)^s}\left( \frac{1+|x|}{|x-x_0|} \right)^{s} \right)\left( \frac{1+|x|}{|x-x_0|} \right)^{n+s}\frac{dx}{(1+|x|)^{n+s}}\\
&\le \left( 1+\frac{1+|x_0|}{R} \right)^{n+sq}\int_{\mr^{n} \setminus B_{R}(x_{0})}g\left( \frac{|u(x)|}{(1+|x|)^s} \right)\frac{dx}{(1+|x|)^{n+s}} < \infty.
\end{align*}
To show the converse relation, choose two different ponts $x_1$, $x_2$ with $|x_1|>1$, $|x_2|>1$, and 
let $0<R \leq  \frac{|x_1-x_2|}{4}$. Then we find
\begin{align*}
\frac{|x-x_i|}{1+|x|} \le 1+\frac{|x_i|-1}{1+|x|} \le |x_i|,\  i=1,2.
\end{align*}
Therefore we can estimate as above that
\begin{align*}
&\int_{\mr^{n}}g\left( \frac{|u(x)|}{(1+|x|)^{s}} \right) \frac{dx}{(1+|x|)^{n+s}}\\
&\quad \le \int_{\mr^{n} \setminus B_{R}(x_1)}g\left( \frac{|u(x)|}{(1+|x|)^{s}} \right) \frac{dx}{(1+|x|)^{n+s}} + \int_{\mr^{n} \setminus B_{R}(x_2)}g\left( \frac{|u(x)|}{(1+|x|)^{s}} \right) \frac{dx}{(1+|x|)^{n+s}}\\
&\quad \le |x_1|^{n+sq}\int_{\mr^{n} \setminus B_{R}(x_{1})}g\left( \frac{|u(x)|}{|x-x_{1}|^s} \right)\frac{dx}{|x-x_{1}|^{n+s}}\\
&\quad \quad + |x_2|^{n+sq}\int_{\mr^{n} \setminus B_{R}(x_{2})}g\left( \frac{|u(x)|}{|x-x_{2}|^s} \right)\frac{dx}{|x-x_{2}|^{n+s}}<\infty.
\end{align*}

\begin{remark}
Observe that
\begin{align}
\label{cor.tail}
R^{s}g^{-1}\left( R^{s}\mathrm{Tail}(u;x_{0},R) \right)
= R^{s}g^{-1}\left( R^{s}\int_{\mr^{n} \setminus B_{R}(x_{0})}g\left( \frac{|u(x)|}{|x-x_{0}|^s} \right)\frac{dx}{|x-x_{0}|^{n+s}} \right).
\end{align}
For the case of $g(t)=t^{p-1}$, \eqref{cor.tail} is
reduced to
\begin{align*}%\label{wsp.tail}
\left[ R^{sp}\int_{\mr^{n} \setminus B_{R}(x_{0})}\frac{|u(x)|^{p-1}}{|x-x_{0}|^{n+sp}}\,dx \right]^{\frac{1}{p-1}},
\end{align*}
which is the tail used in \cite{DKP2}.
In this paper we use \eqref{tail} instead of \eqref{cor.tail} for simplicity.
\end{remark}

\begin{remark} \,
\begin{enumerate}
\item 
Note that $W^{s,G}(\mr^{n}) \subset \mathbb{W}^{s,G}(\Omega) \cap L^{g}_{s}(\mr^{n})$.

\item
Let $\psi$ be a Young function satisfying $g(t) \le c\psi(t)$ for $t \ge t_{0}$, where $c$ and $t_{0}$ are some positive constants.
If $u \in L^{\psi}(\mr^{n})$ or $u \in L^{\psi}(B_R(0)) \cap L^{\infty}(\mr^{n} \setminus B_R(0))$, then $u \in L_{s}^{g}(\mr^{n})$.
\end{enumerate}
\end{remark}

\section{Auxiliary Estimates}\label{sec3}

In this section we derive two estimates for the weak solutions to \eqref{pde} that play essential roles in the proof of the main theorem.
The first one is a Caccioppoli type estimate. A similar Caccioppoli type estimate in the Orlicz setting can be also found in \cite{CKW1}.
\begin{proposition}
\label{lem.caccio}(Caccioppoli type estimate)
Let $u \in \mathbb{W}^{s,G}(\Omega) \cap L^{g}_{s}(\mr^{n})$ be a weak solution to \eqref{pde}.
Then for any $k\ge 0$, $B_{r} \equiv B_{r}(x_0) \Subset \Omega$ and $\phi \in C_{0}^{\infty}(B_{r})$
with $0\le \phi\le 1$, we have
\begin{align}\label{caccio}
\begin{split}
&\int_{B_{r}}\int_{B_{r}}G\left(\frac{|w_{\pm}(x)-w_{\pm}(y)|}{|x-y|^s}\right)\min{\{ \phi^{q}(x),\phi^{q}(y) \}}\frac{dxdy}{|x-y|^{n}}\\
&\quad\le c\int_{B_{r}}\int_{B_{r}}G\left(\frac{|\phi(x)-\phi(y)|}{|x-y|^s} \max \{w_{+}(x),w_{+}(y)\} \right)\frac{dxdy}{|x-y|^{n}}\\
&\qquad +c\int_{B_{r}}w_{\pm}(x)\phi^{q}(x)\,dx\left( \sup_{y\, \in\, \mathrm{supp}\,\phi}\int_{\mr^{n} \setminus B_{r}}g\left(\frac{w_{\pm}(x)}{|x-y|^{s}}\right)\frac{dx}{|x-y|^{n+s}}\right),
\end{split}
\end{align}
where $w_{\pm}:=(u-k)_{\pm}=\max\{\pm(u-k),0\}$ and $c>0$ depends on $n,s,p,q,\lambda$ and $\Lambda$.
\end{proposition}

\begin{proof}
We only consider $w_{+}$, as the same argument can apply to $w_{-}$.
Take $\eta:=w_{+}\phi^{q} \in \mathbb{W}^{s,p}(\Omega)$ as a test function to find
\begin{align}\label{c.1}
\begin{split}
0 &= \int_{\mr^n}\int_{\mr^n}g\left(\frac{|u(x)-u(y)|}{|x-y|^s}\right)\frac{u(x)-u(y)}{|u(x)-u(y)|}(\eta(x)-\eta(y))K(x,y)\frac{dxdy}{|x-y|^{s}}\\
&= \int_{B_{r}}\int_{B_{r}}g\left(\frac{|u(x)-u(y)|}{|x-y|^s}\right)\frac{u(x)-u(y)}{|u(x)-u(y)|}(\eta(x)-\eta(y))K(x,y)\frac{dxdy}{|x-y|^{s}}\\
&\quad + 2\int_{\mr^n \setminus B_{r}}\int_{B_{r}}g\left(\frac{|u(x)-u(y)|}{|x-y|^s}\right)\frac{u(x)-u(y)}{|u(x)-u(y)|}\eta(x)K(x,y)\frac{dxdy}{|x-y|^{s}}\\
&=:\rom{1} + \rom{2}.
\end{split}
\end{align}
Note that $\eta(x)=0$ for $x \in B_r \cap \{ u(x)<k \}$. We divide the latter part of the proof into two steps.

\textit{Step 1.} In this step we derive an estimate in terms of $w_{+}$ from \eqref{c.1}.
We first consider the integrand of $\rom{1}$ with respect to the measure, $K(x,y)\frac{dxdy}{|x-y|^{s}}$.
In the case when $u(x) \ge u(y)$ for $x,y \in B_{r}$,
we have
\begin{align}\label{c.rom1.c1}
\begin{split}
&g\left(\frac{|u(x)-u(y)|}{|x-y|^s}\right)\frac{u(x)-u(y)}{|u(x)-u(y)|}(\eta(x)-\eta(y))\\
&= g\left(\frac{u(x)-u(y)}{|x-y|^s}\right)(\eta(x)-\eta(y))\\
&=
\begin{cases}
\ds g\left(\frac{w_{+}(x)-w_{+}(y)}{|x-y|^s}\right)(\eta(x)-\eta(y)) &\  \mbox{if} \ \ u(x) \ge u(y) \ge k, \\
\ds g\left(\frac{u(x)-u(y)}{|x-y|^s}\right)\eta(x) &\ \mbox{if} \ \ u(x) \ge k > u(y), \\
\ds 0 &\ \mbox{if} \ \ k > u(x) \ge u(y)
\end{cases}\\
&\ge g\left( \frac{w_{+}(x)-w_{+}(y)}{|x-y|^s} \right)(\eta(x)-\eta(y))\\
&= g\left( \frac{|w_{+}(x)-w_{+}(y)|}{|x-y|^s} \right)\frac{w_{+}(x)-w_{+}(y)}{|w_{+}(x)-w_{+}(y)|}(\eta(x)-\eta(y)).
\end{split}
\end{align}
On the other hand, in the case when $u(x) < u(y)$ for $x,y\in B_r$, we exchange the roles of $x$ and $y$ in \eqref{c.rom1.c1} to obtain the same result.
Then we recall the assumption \eqref{kernel} to get
\begin{align}\label{c.rom1}
\rom{1} \ge \lambda\int_{B_{r}}\int_{B_{r}}g\left( \frac{|w_{+}(x)-w_{+}(y)|}{|x-y|^s} \right)\frac{w_{+}(x)-w_{+}(y)}{|w_{+}(x)-w_{+}(y)|}(\eta(x)-\eta(y))\frac{dxdy}{|x-y|^{n+s}}.
\end{align}
Next, let us consider $\rom{2}$. Note that
\begin{align*}
g\left(\frac{|u(x)-u(y)|}{|x-y|^s}\right)\frac{u(x)-u(y)}{|u(x)-u(y)|}\eta(x) \ge
\begin{cases}
-g\left(\frac{w_+(y)}{|x-y|^s}\right)\eta(x) &\ \ \text{if}\ \ u(y)>u(x)\ge k,\\
0 & \ \ \text{if}\ \ \text{otherwise}.
\end{cases}
\end{align*}
Inserting this inequality into $\rom{2}$, we deduce
\begin{align}\label{c.rom2}
\rom{2} \ge -2\Lambda\int_{\mr^n \setminus B_{r}}\int_{B_{r}}g\left( \frac{w_{+}(y)}{|x-y|^s} \right)\eta(x)\frac{dxdy}{|x-y|^{n+s}}.
\end{align}
%In both cases, we conclude that \eqref{c.rom2} holds. 
We then combine \eqref{c.1}, \eqref{c.rom1}, and \eqref{c.rom2} to discover
\begin{align}
\label{c.sub}
\begin{split}
&\int_{B_{r}}\int_{B_{r}}g\left( \frac{|w_{+}(x)-w_{+}(y)|}{|x-y|^s} \right)\frac{w_{+}(x)-w_{+}(y)}{|w_{+}(x)-w_{+}(y)|}(\eta(x)-\eta(y))\frac{dxdy}{|x-y|^{n+s}}\\
&\le \frac{2\Lambda}{\lambda}\int_{\mr^n \setminus B_{r}}\int_{B_{r}}g\left( \frac{w_{+}(y)}{|x-y|^s} \right)\eta(x)\frac{dxdy}{|x-y|^{n+s}}.
\end{split}
\end{align}

\textit{Step 2.} Set
\begin{align*}
\rom{3}:=\int_{B_{r}}\int_{B_{r}}g\left( \frac{|w_{+}(x)-w_{+}(y)|}{|x-y|^s} \right)\frac{w_{+}(x)-w_{+}(y)}{|w_{+}(x)-w_{+}(y)|}(\eta(x)-\eta(y))\frac{dxdy}{|x-y|^{n+s}}
\end{align*}
and
\begin{align*}
\rom{4}:=\int_{\mr^n \setminus B_{r}}\int_{B_{r}}g\left( \frac{w_{+}(y)}{|x-y|^s} \right)\eta(x)\frac{dxdy}{|x-y|^{n+s}}.
\end{align*}
Then we see from \eqref{c.sub} that $\rom{3}\le \frac{2\Lambda}{\lambda} \rom{4}$.

To estimate $\rom{3}$, we first look at the integrand of $\rom{3}$ with respect to the measure $\frac{dxdy}{|x-y|^{n}}$.
Consider the following three cases:
\begin{enumerate}
\item $w_{+}(x)>w_{+}(y)$ and  $\phi(x) \le \phi(y)$,
\item $w_{+}(x)>w_{+}(y)$ and  $\phi(x) > \phi(y)$,
\item $w_{+}(x) \le w_{+}(y)$.
\end{enumerate}
\noindent In the case (1), we have
\begin{align}
\label{c.rom3.c1.1}
\begin{split}
&g\left(\frac{|w_{+}(x)-w_{+}(y)|}{|x-y|^s}\right)\frac{w_{+}(x)-w_{+}(y)}{|w_{+}(x)-w_{+}(y)|}\frac{w_+(x)\phi^q(x)-w_+(y)\phi^q(y)}{|x-y|^s}\\
&= g\left(\frac{w_{+}(x)-w_{+}(y)}{|x-y|^s}\right)\frac{w_{+}(x)-w_{+}(y)}{|x-y|^s}\phi^{q}(y) - g
\left(\frac{w_{+}(x)-w_{+}(y)}{|x-y|^s}\right)\frac{\phi^{q}(y)-\phi^{q}(x)}{|x-y|^s}w_{+}(x)\\
&\ge p G\left(\frac{w_{+}(x)-w_{+}(y)}{|x-y|^s}\right)\phi^{q}(y) - qg\left(\frac{w_{+}(x)-w_{+}(y)}{|x-y|^s}\right)
\phi^{q-1}(y)\frac{\phi(y)-\phi(x)}{|x-y|^s}w_{+}(x),
\end{split}
\end{align}
where we have used \eqref{growth} and the following elementary inequality
\begin{align*}
\phi^{q}(y)-\phi^{q}(x) \le q\phi^{q-1}(y)(\phi(y)-\phi(x)).
\end{align*}
We further estimate the second term on right hand side of \eqref{c.rom3.c1.1}. By using \eqref{Young} and  \eqref{Young1}, we have that for $\ep \in (0,1)$,
\begin{align}\label{c.rom3.y}
\begin{split}
&g\left(\frac{w_{+}(x)-w_{+}(y)}{|x-y|^s}\right)\phi^{q-1}(y)\frac{\phi(y)-\phi(x)}{|x-y|^s}w_{+}(x)\\
&\le \ep G^{*}\left(g\left( \frac{w_{+}(x)-w_{+}(y)}{|x-y|^s} \right)\phi^{q-1}(y)\right)  + c(\ep)G\left( \frac{\phi(y)-\phi(x)}{|x-y|^s}w_{+}(x) \right)\\
&\le \ep(q-1) G\left(\frac{w_{+}(x)-w_{+}(y)}{|x-y|^s}\right)\phi^{q}(y) + c(\ep)G\left( \frac{\phi(y)-\phi(x)}{|x-y|^s}w_{+}(x) \right).
\end{split}
\end{align}
For the last inequality, we have used \eqref{G*ineq} with
$a=\phi^{q-1}(y)\le 1$. Choosing $\ds \ep = \min \left\{
\frac{p}{2(q-1)}, \frac{1}{2} \right\}$ and plugging
\eqref{c.rom3.y} into \eqref{c.rom3.c1.1}, we discover
\begin{align}\label{c.rom3}
\begin{split}
&g\left(\frac{|w_{+}(x)-w_{+}(y)|}{|x-y|^s}\right)\frac{w_{+}(x)-w_{+}(y)}{|w_{+}(x)-w_{+}(y)|}\frac{w_+(x)\phi^q(x)-w_+(y)\phi^q(y)}{|x-y|^s}\\
&\ge \frac{p}{2}G\left(\frac{|w_{+}(x)-w_{+}(y)|}{|x-y|^s}\right)\min \{ \phi^{q}(x), \phi^{q}(y) \} - cG\left( \frac{|\phi(x)-\phi(y)|}{|x-y|^s}\max \{ w_{+}(x), w_{+}(y) \} \right).
\end{split}\end{align}
In the case (2), %that $w_{+}(x)>w_{+}(y)$ and  $\phi(x) > \phi(y)$, 
we use \eqref{growth} to have
\begin{align*}
&g\left(\frac{|w_{+}(x)-w_{+}(y)|}{|x-y|^s}\right)\frac{w_{+}(x)-w_{+}(y)}{|w_{+}(x)-w_{+}(y)|}\frac{w_+(x)\phi^q(x)-w_+(y)\phi^q(y)}{|x-y|^s}\\
&\ge g\left(\frac{w_{+}(x)-w_{+}(y)}{|x-y|^s}\right)\frac{w_{+}(x)-w_{+}(y)}{|x-y|^s}\phi^{q}(x)\ge pG\left(\frac{|w_{+}(x)-w_{+}(y)|}{|x-y|^s}\right)\min \{ \phi^{q}(x), \phi^{q}(y) \}.
\end{align*}
Therefore, we also obtain the estimate \eqref{c.rom3} in this case. Moreover, since the integrand is invariant with the exchanging of $x$ and $y$,
we again have the estimate \eqref{c.rom3} in the case (3). Consequently, we obtain
\begin{align*}
\begin{split}
\rom{3} &\ge c\int_{B_{r}}\int_{B_{r}}G\left(\frac{|w_{+}(x)-w_{+}(y)|}{|x-y|^s}\right)\min{ \{ \phi^{q}(x), \phi^{q}(x) \}}\frac{dxdy}{|x-y|^{n+s}}\\
& \quad - c\int_{B_{r}}\int_{B_{r}}G\left( \frac{|\phi(x)-\phi(y)|}{|x-y|^s}\max \{ w_{+}(x), w_{+}(y) \} \right)\frac{dxdy}{|x-y|^{n+s}}.
\end{split}
\end{align*}
To estimate $\rom{4}$, we first use Fubini's theorem to find
\begin{align*}%\label{c.rom4}
\begin{split}
\rom{4} &= \int_{\mr^n \setminus B_{r}}\int_{B_{r}}g\left( \frac{w_{+}(y)}{|x-y|^s} \right)\eta(x)\frac{dxdy}{|x-y|^{n+s}}\\
&\le \int_{B_{r}}w_{+}(x)\phi^{q}(x)\,dx\left( \sup_{y\, \in\, \mathrm{supp}\,\phi}\int_{\mr^{n} \setminus B_{r}}g\left(\frac{w_{+}(y)}{|x-y|^{s}}
\right)\frac{dy}{|x-y|^{n+s}}\right).
\end{split}
\end{align*}
Hence we obtain \eqref{caccio}, as $\rom{3} \le  c \rom{4}$.
\end{proof}

\begin{remark}
\label{rmk.caccio}
In Proposition \ref{lem.caccio}, the estimate \eqref{caccio} for $w_+$ (resp. $w_-$) still holds true when $u$ is a weak subsolution (resp. supersolution)
to \eqref{pde}.
\end{remark}

The second one to be derived is a logarithmic estimate. This will be used in the proof of the decay estimate for
the oscillation of weak solutions, Lemma~\ref{lem.osc}. We need the following elementary inequality.
\begin{lemma}(\cite[Lemma 3.1]{DKP2}
\label{lo.lem})
Let $q \ge 1$ and $\ep \in (0,1]$. Then
\begin{align*}
|a|^{q} \le (1+c_q\ep)|b|^{q}+(1+c_q\ep)\ep^{1-q}|a-b|^{q}
\end{align*}
for every $a, b \in \mr^{n}$. Here $c_q > 0$ depends on $n$ and $q$.
\end{lemma}

\begin{proposition}
\label{lem.log}(Logarithmic estimate)
Let $u \in \mathbb{W}^{s,p}(\Omega) \cap L_{s}^{g}(\mr^{n})$ be a weak supersolution to \eqref{pde} with
$u \ge 0$ in $B_{R}(x_0) \subset \Omega$. Then for any $d>0$ and $0<r<\frac{R}{2}$, we have
\begin{align}\label{lo.in}
\int_{B_{r}}\int_{B_{r}}|\log{(u(x)+d)}-\log{(u(y)+d)}|\frac{dxdy}{|x-y|^n} \le cr^n + c\frac{r^{n+s}}{g(d/r^{s})}\mathrm{Tail}(u_{-};x_0,R)
\end{align}
for some $c=c(n,s,p,q,\lambda,\Lambda)>0$, where $B_r=B_r(x_0)$ and $B_R(x_0)=B_R$. In addition, we have the estimate
\begin{align}\label{lo}
\int_{B_{r}}|h-(h)_{B_r}|\,dx \le cr^n\left[ 1 + \frac{r^s}{g(d/r^{s})}\mathrm{Tail}(u_{-};x_0,R) \right],
\end{align}
where
\begin{align*}
h:= \min\left\{ (\log(a+d) - \log(u+d))_{+}, \log{b} \right\}, \quad a>0 \ \ \text{and}\ \ b>1.
\end{align*}
\end{proposition}

\begin{proof}
Write $v(x):=u(x)+d$ and fix a cut-off function $\ds \phi \in C^{\infty}_{0}(B_{3r/2})$ such that $0 \le \phi \le 1$, $|D\phi| \le \frac{4}{r}$ and $\phi \equiv 1$ in $B_{r}$.
Since $\frac{v}{G(v/r^s)}$ is nonnegative in $B_R$ and  belongs to $\mathbb{W}^{s,p}(\Omega)$,
we can take $\eta=\frac{v\phi^{q}}{G(v/r^s)}$ as a test function to find
\begin{align}\label{lo.sp}
\begin{split}
0&\le \int_{B_{2r}}\int_{B_{2r}}g\left(\frac{|v(x)-v(y)|}{|x-y|^s}\right)\frac{v(x)-v(y)}{|v(x)-v(y)|}(\eta(x)-\eta(y))K(x,y)\frac{dxdy}{|x-y|^{s}}\\
&\quad +2\int_{\mr^{n} \setminus B_{2r}}\int_{B_{2r}}g\left(\frac{|u(x)-u(y)|}{|x-y|^s}\right)\frac{u(x)-u(y)}{|u(x)-u(y)|}\eta(x)K(x,y)\frac{dxdy}{|x-y|^{s}}\\
&=:\rom{1}+\rom{2}.
\end{split}
\end{align}
We define
\begin{align*}
F=F(x,y):=g\left(\frac{|v(x)-v(y)|}{|x-y|^s}\right)\frac{v(x)-v(y)}{|v(x)-v(y)|}\frac{\eta(x)-\eta(y)}{|x-y|^{s}}, \quad x,y\in B_{2r}.
\end{align*}
Note that $F(x,y)=F(y,x)$.
We divide the remaining proof into four steps.

\smallskip
\textit{Step 1.}
We first assume that $v(y) \le v(x) \le 2v(y)$ for $x,y\in B_{2r}$ to assert that
\begin{align}\label{estimateF}
F(x,y) \le - \tilde c (\log v(x) -\log v(y))\phi(x)^q+c \left( \frac{|x-y|}{r} \right)^{s}+
+c \left( \frac{|x-y|}{r} \right)^{(1-s)p}
\end{align}
for some small constant $\tilde c>0$ and large constant $c>0$ depending on $n,p$ and $q$.
To prove this, let us suppose $\phi(x) \ge \phi(y)$. By the definition of $\eta$, we get
\begin{align}\label{lo.rom1.sp}
\begin{split}
F(x,y)&= g\left(\frac{v(x)-v(y)}{|x-y|^s}\right)\left(\frac{v(x)}{G(v(x)/r^s)}-\frac{v(y)}{G(v(y)/r^s)}\right)\frac{\phi^{q}(x)}{|x-y|^{s}}\\
&\qquad + g\left(\frac{v(x)-v(y)}{|x-y|^s}\right)\frac{v(y)}{G(v(y)/r^{s})}\frac{\phi^{q}(x)-\phi^{q}(y)}{|x-y|^{s}}\\
&=: F_{1}(x,y)+F_{2}(x,y).
\end{split}
\end{align}
Applying Mean Value Theorem to the mappings $t\mapsto \frac{t}{G(t/r^s)}$ for $v(y)\le t \le v(x)$,
and $t\mapsto t^q$ for $\phi(y) \le t \le \phi(x)$, respectively, we have
\begin{align}\label{lo.mvt1}
\frac{v(x)}{G(v(x)/r^s)}-\frac{v(y)}{G(v(y)/r^s)} \le -(p-1)\frac{v(x)-v(y)}{G(v(x)/r^{s})}
\end{align}
and
\begin{align}\label{lo.mvt2}
\phi^{q}(x)-\phi^{q}(y) \le q\phi^{q-1}(x)(\phi(x)-\phi(y)).
\end{align}
For the inequality \eqref{lo.mvt1}, we have used \eqref{growth}.
Putting \eqref{lo.mvt1} into $F_{1}$ and using \eqref{growth} and
the fact that $v(x) \le 2v(y)$, we have
\begin{align}\label{lo.rom1.c1}
\begin{split}F_{1} & \le -(p-1)g\left(\frac{v(x)-v(y)}{|x-y|^s}\right)\frac{v(x)-v(y)}{|x-y|^{s}}\frac{\phi^{q}(x)}{G(v(x)/r^s)}\\
&  \le  -c_1 G\left(\frac{v(x)-v(y)}{|x-y|^s}\right)\frac{\phi^{q}(x)}{G(v(y)/r^s)}
\end{split}
\end{align}
for some small constant $c_1=c_1(p,q)>0$.

We use \eqref{lo.mvt2} and recall \eqref{Young} with $\ds \ep=\min \left\{ \frac{c_{1}}{2q(q-1)},
\frac{1}{2} \right\}$ and \eqref{Young1}, to discover
\begin{align}\label{lo.rom1.c2}
\begin{split}
F_{2} & \le qg\left(\frac{v(x)-v(y)}{|x-y|^s}\right)\phi^{q-1}(x)\frac{\phi(x)-\phi(y)}{|x-y|^s}\frac{v(y)}{G(v(y)/r^s)}\\
& \le q\left[ \ep (q-1) G\left(\frac{v(x)-v(y)}{|x-y|^s}\right)\phi^{q}(x)+\ep^{1-q} G\left( \frac{\phi(x)-\phi(y)}{|x-y|^s}v(y) \right) \right]\frac{1}{G(v(y)/r^s)}\\
& \le \left[\frac{c_1}{2}G\left(\frac{v(x)-v(y)}{|x-y|^s}\right)\phi^{q}(x)+c G \left( \frac{\phi(x)-\phi(y)}{|x-y|^s}v(y) \right) \right]\frac{1}{G(v(y)/r^s)}.
\end{split}
\end{align}
We then combine \eqref{lo.rom1.sp}, \eqref{lo.rom1.c1}, and \eqref{lo.rom1.c2} and use the fact that
$|D\phi| \le \frac{4}{r}$ and $|x-y|\le 4r$ for $x,y\in B_{2r}$, to obtain \eqref{estimateF}.

We next suppose $\phi(x) < \phi(y)$. Using \eqref{lo.mvt1} and \eqref{growth}, we have
\begin{align*}
F (x,y)&= g\left(\frac{v(x)-v(y)}{|x-y|^s}\right)\left(\frac{\phi^{q}(x)v(x)}{G(v(x)/r^s)}-\frac{\phi^{q}(y)v(y)}
{G(v(y)/r^s)}\right)\frac{1}{|x-y|^{s}}\\
&\le g\left(\frac{v(x)-v(y)}{|x-y|^s}\right)\left(\frac{v(x)}{G(v(x)/r^s)}-\frac{v(y)}{G(v(y)/r^s)}\right)\frac{\phi^{q}(y)}{|x-y|^{s}}\\
&\le -cg\left(\frac{v(x)-v(y)}{|x-y|^s}\right)\frac{v(x)-v(y)}{|x-y|^{s}}\frac{\phi^{q}(y)}{G(v(x)/r^s)}\\
&\le -cG\left(\frac{v(x)-v(y)}{|x-y|^s}\right)\frac{\phi^{q}(x)}{G(v(y)/r^s)}.
\end{align*}
Therefore we also have
\begin{align}\label{F1}
F(x,y)\le - c G\left(\frac{v(x)-v(y)}{|x-y|^s}\right)\frac{\phi^{q}(x)}{G(v(y)/r^s)}
+c \left( \frac{|x-y|}{r} \right)^{(1-s)p}.
\end{align}
In addition, by Mean Value Theorem,
\begin{align*}
\log v(x) -\log v(y) &\le  \frac{v(x)-v(y)}{v(y)} =\frac{(v(x)-v(y))/|x-y|^s}{v(y)/r^s} \frac{|x-y|^s}{r^s}\\
&\le  \left\{G\left(\frac{v(x)-v(y)}{|x-y|^s}\right) \frac{1}{G(v(y)/r^s)} +1\right\} \frac{|x-y|^s}{r^s}\\
&\le c G\left(\frac{v(x)-v(y)}{|x-y|^s}\right) \frac{1}{G(v(y)/r^s)} + \frac{|x-y|^s}{r^s}.
\end{align*}
For the second inequality we have used the fact that $\frac{G(t)}{t}$ is increasing for $t$.
This estimate and \eqref{F1} imply finally \eqref{estimateF}.

\smallskip
\textit{Step 2.}
We now assume that $v(x) > 2v(y)$ for $x,y\in B_{2r}$ to claim that
\begin{align}\label{estimateF1}
F(x,y)\le -\tilde c (\log v(x)-\log v(y)) \phi^{q}(y)+c\left(\frac{|x-y|}{r}\right)^{s(p-1)} +c\left(\frac{|x-y|}{r}\right)^{(1-s)q}
\end{align}
for some small constant $\tilde c>0$ and large constant $c>0$ depending on $n,p$ and $q$.
To this end, we recall the definition of $\eta$ to see that
\begin{align*}
F(x,y)&= g\left(\frac{v(x)-v(y)}{|x-y|^s}\right)\left(\frac{v(x)}{G(v(x)/r^s)}-\frac{v(y)}{G(v(y)/r^s)}\right)\frac{\phi^{q}(y)}{|x-y|^{s}}\\
&\quad +g\left(\frac{v(x)-v(y)}{|x-y|^s}\right)\frac{v(x)}{G(v(x)/r^{s})}\frac{\phi^{q}(x)-\phi^{q}(y)}{|x-y|^{s}}\\
&=: F_{3}(x,y)+F_{4}(x,y).
\end{align*}
Since $\frac{t}{G(t)}$ is decreasing for $t$, we have
\begin{align*}
F_{3} &\le g\left( \frac{v(x)-v(y)}{|x-y|^s} \right)\left( \frac{2v(y)}{G(2v(y)/r^s)}-\frac{v(y)}{G(v(y)/r^s)} \right) \frac{\phi^{q}(y)}{|x-y|^{s}}\\
&= g\left( \frac{v(x)-v(y)}{|x-y|^s} \right)\frac{v(y)}{G(v(y)/r^s)}\left( 2\frac{G(v(y)/r^{s})}{G(2v(y)/r^{s})} - 1 \right)\frac{\phi^{q}(y)}{|x-y|^{s}}\\
&\le -\left( 1 - \frac{1}{2^{p-1}} \right)g\left( \frac{v(x)-v(y)}{|x-y|^s} \right)\frac{v(y)}{G(v(y)/r^s)}\frac{\phi^{q}(y)}{|x-y|^{s}}.
\end{align*}
On the other hand, in light of Lemma \ref{lo.lem}, we find that for $\ep \in (0,1)$,
\begin{align*}
F_{4} &\le c_q\ep g\left(\frac{v(x)-v(y)}{|x-y|^s}\right)\frac{v(x)}{G(v(x)/r^s)}\frac{\phi^{q}(y)}{|x-y|^{s}}\\
&\qquad +c\ep^{1-q}g\left(\frac{v(x)-v(y)}{|x-y|^s}\right)\frac{v(x)}{G(v(x)/r^s)}\frac{|\phi(x)-\phi(y)|^{q}}{|x-y|^{s}}.
\end{align*}
In addition, using the fact that $\frac{t}{G(t)}$ is decreasing for $t$, $2v(y) < v(x)$, $|D\phi| \le \frac{c}{r}$, $|x-y|\le 4r$ for $x,y\in B_{2r}$ and $tg(t)\le q G(t)$, we discover
\begin{align*}
F_{4} \le c_q\ep g\left(\frac{v(x)-v(y)}{|x-y|^s}\right)\frac{v(y)}{G(v(y)/r^s)}\frac{\phi^{q}(y)}{|x-y|^{s}}+c\ep^{1-q}\left(\frac{|x-y|}{r}\right)^{(1-s)q}.
\end{align*}
We then choose $\ds \ep=\min \left\{ \frac{1}{2c_{q}}\left( 1 - \frac{1}{2^{p-1}} \right), \frac{1}{2} \right\}$, and combine the above estimates to discover
\begin{align*}
F \le - c g\left(\frac{v(x)-v(y)}{|x-y|^s}\right)\frac{v(y)}{G(v(y)/r^s)}\frac{\phi^{q}(y)}{|x-y|^{s}}+c\left(\frac{|x-y|}{r}\right)^{(1-s)q}.
\end{align*}
Note
\begin{align*}
\frac{v(y)}{G(v(y)/r^s)}\frac{1}{|x-y|^{s}} \ge \frac{1}{4^s}\frac{v(y)/r^s}{G(v(y)/r^s)} \ge \frac{p}{4^s} \frac{1}{g(v(y)/r^s)}
\end{align*}
to have
\begin{align}\label{F2}
F(x,y)\le -c g\left(\frac{v(x)-v(y)}{|x-y|^s}\right)\frac{\phi^{q}(y)}{g(v(y)/r^s)} +c\left(\frac{|x-y|}{r}\right)^{(1-s)q}.
\end{align}
Moreover, since $v(x) > 2v(y)$,
\begin{align*}
\log v(x) -\log v(y) \le \log 2(v(x)-v(y)) -\log v(y) \le c \left(\frac{2(v(x)-v(y))}{v(y)} \right)^{p-1},
\end{align*}
where we have used the fact that $\log t <  \frac{t^{p-1}}{p-1}$. Note that
\begin{align*}
\frac{g(s)}{s^{p-1}} \le q \frac{G(s)}{s^p} \le  q \frac{G(t)}{t^p}  \le \frac{q}{p} \frac{g(t)}{t^{p-1}} \quad \mbox{for any }\ t\ge s\ge 0,
\end{align*}
to discover
\begin{align*}
\log v(x) -\log v(y) & \le c \left(\frac{(v(x)-v(y))/|x-y|^s}{v(y)/r^s} \frac{|x-y|^s}{r^s} \right)^{p-1}\\
&\le c g\left(\frac{v(x)-v(y)}{|x-y|^s}\right)\frac{1}{g(v(y)/r^s)} +c \left( \frac{|x-y|}{r} \right)^{s(p-1)}.
\end{align*}
This and \eqref{F2} imply the estimate \eqref{estimateF1}.

\smallskip
\textit{Step 3.}
We next estimate $\rom{1}$ in \eqref{lo.sp}.
We recall \eqref{estimateF} when $v(y)\le v(x)\le 2v(y)$, and \eqref{estimateF1} when $v(x)>2v(y)$, and use the fact $F(x,y)=F(y,x)$,
to discover that for every $x,y\in B_{2r}$,
\begin{align*}
F(x,y) &\le  - \tilde c \, |\log v(x) -\log v(y)| \min\{\phi(x),\phi(y)\}^q\\
&\qquad +c \left( \frac{|x-y|}{r} \right)^{s}
+c \left( \frac{|x-y|}{r} \right)^{(1-s)p}+c\left(\frac{|x-y|}{r}\right)^{s(p-1)}.
\end{align*}
Then since $\phi\equiv 1$ in $B_r$ and $K(x,y)$ satisfies  \eqref{kernel}, we have
\begin{align}\label{estimateI}\begin{split}
\rom{1}&\le -\frac{\tilde c}{\lambda}  \int_{B_{r}}\int_{B_{r}}|\log{v(x)}-\log{v(y)}|\frac{dxdy}{|x-y|^n}  \\
&\quad  + c  \int_{B_{2r}}\int_{B_{2r}} \left[ \left( \frac{|x-y|}{r} \right)^{s} + \left(\frac{|x-y|}{r} \right)^{(1-s)p}+c\left(\frac{|x-y|}{r}\right)^{s(p-1)} \right]\frac{dxdy}{|x-y|^n}\\
&\le -\frac{\tilde c}{\lambda}  \int_{B_{r}}\int_{B_{r}}|\log{v(x)}-\log{v(y)}|\frac{dxdy}{|x-y|^n} +c r^n.
\end{split}
\end{align}

\smallskip
\textit{Step 4.}
We next estimate $\rom{2}$. Observe that for $x \in B_{R}$ and $y \in \mr^{n}$,
\begin{align*}
g\left(\frac{|u(x)-u(y)|}{|x-y|^{s}}\right)\frac{u(x)-u(y)}{|u(x)-u(y)|} &\le g\left(\frac{(u(x)-u(y))_{+}}{|x-y|^{s}}\right)\le c\left[ g\left(\frac{u(x)}{|x-y|^s}\right)+g\left(\frac{u(y)_{-}}{|x-y|^s}\right) \right].
\end{align*}
Recalling $\mathrm{supp}\, \phi \subset B_{3r/2}$, we have
\begin{align*}%\label{lo.rom2.sp}
\begin{split}
\rom{2} &\le c\int_{\mr^{n} \setminus B_{2r}}\int_{B_{3r/2}}g\left(\frac{(u(x)-u(y))_{+}}{|x-y|^s}\right)\frac{r^{s}}{g(v(x)/r^{s})}\frac{dxdy}{|x-y|^{n+s}}\\
&\le c\int_{B_{R} \setminus B_{2r}}\int_{B_{3r/2}}g\left(\frac{(u(x)-u(y))_{+}}{|x-y|^s}\right)\frac{r^{s}}{g(v(x)/r^{s})}\frac{dxdy}{|x-y|^{n+s}}\\
&\quad +c\int_{\mr^{n} \setminus B_{R}}\int_{B_{3r/2}}g\left(\frac{u(x)}{|x-y|^s}\right)\frac{r^{s}}{g(v(x)/r^{s})}\frac{dxdy}{|x-y|^{n+s}}\\
&\quad +c\int_{\mr^{n} \setminus B_{R}}\int_{B_{3r/2}}g\left(\frac{u(y)_{-}}{|x-y|^s}\right)\frac{r^{s}}{g(v(x)/r^{s})}\frac{dxdy}{|x-y|^{n+s}}\\
&=:\rom{2}_{1}+\rom{2}_{2}+\rom{2}_{3}.
\end{split}
\end{align*}
Since $u \ge 0$ in $B_{R}$ and  $v=u+d$, we see that
\begin{align*}
(u(x)-u(y))_{+} \le v(x)
 \ \mbox{ and } \
u(x) \le v(x), \quad x,y\in B_R.
\end{align*}
Thus
\begin{align*}
\rom{2}_{1} &\le c\int_{B_{R} \setminus B_{2r}}\int_{B_{3r/2}}g\left(\frac{(u(x)-u(y))_{+}}{r^s}\right)\frac{r^{s}}{g(v(x)/r^{s})}\frac{dxdy}{|x-y|^{n+s}}\\
&\le cr^s \int_{B_{R} \setminus B_{2r}}\int_{B_{3r/2}}\frac{dxdy}{|x-y|^{n+s}}  \le cr^s \int_{B_{3r/2}}\int_{\R^n \setminus B_{r/2}(x)}
\frac{dydx}{|x-y|^{n+s}}  \le c r^n
\end{align*}
and
\begin{align*}
\rom{2}_{2} &\le c\int_{\mr^{n} \setminus B_{R}}\int_{B_{3r/2}}g\left(\frac{u(x)}{r^s}\right) \frac{r^s}{g(v(x)/r^s)}
\frac{dxdy}{|x-y|^{n+s}}\le cr^s \int_{\mr^{n} \setminus B_{2r}}\int_{B_{3r/2}} \frac{dxdy}{|x-y|^{n+s}}  \le c r^n.
\end{align*}
Observing that for any $x \in B_{3r/2}$ and $y \in \mr^{n} \setminus B_{2r}$
\begin{align*}
%\label{lo.dist}
\frac{|y-x_{0}|}{|x-y|} \le 1+\frac{|x-x_{0}|}{|x-y|} \le 1+\frac{3r/2}{2r-(3r/2)} = 4,
\end{align*}
we find
\begin{align*}
\rom{2}_{3} \le c\int_{\mr^{n} \setminus B_{R}}\int_{B_{3r/2}}g\left(\frac{u(y)_{-}}{|y-x_0|^s}\right)\frac{r^{s}}{g(d/r^{s})}\frac{dxdy}{|y-x_0|^{n+s}}\le c\frac{r^{n+s}}{g(d/r^{s})}\mbox{Tail}(u_{-};x_0,R).
\end{align*}
Consequently, we have
\begin{align*}
\rom{2} \le cr^{n} + c\frac{r^{n+s}}{g(d/r^{s})}\mbox{Tail}(u_{-};x_0,R).
\end{align*}
Inserting this estimate and \eqref{estimateI} into \eqref{lo.sp}, we get \eqref{lo.in}.

\smallskip
\textit{Step 5.}
Now we are ready to prove the estimate \eqref{lo}. Observe that
\begin{align*}
\int_{B_{r}}|h-(h)_{B_{r}}|dx \le c\int_{B_r}\int_{B_r}|h(x)-h(y)|\frac{dxdy}{|x-y|^{n}}.
\end{align*}
Since $h(x)$ is a truncation of $\log{v(x)}$, 
\begin{align*}
\int_{B_r}\int_{B_r}|h(x)-h(y)|\frac{dxdy}{|x-y|^{n}} \le c\int_{B_{r}}\int_{B_{r}}|\log{v(x)}-\log{v(y})|\frac{dxdy}{|x-y|^{n}}.
\end{align*}
Combining \eqref{lo.in} and the above inequalities, we finally obtain \eqref{lo}.
\end{proof}

\section{Proof of Theorem~\ref{mainthm}}
\label{sec4}
\subsection{Local Boundedness}
This subsection is devoted to the proof of the local boundedness of weak solutions to \eqref{pde} with the estimate \eqref{lb1} in
Theorem~\ref{mainthm}. Key ingredients of the proof are the Caccioppoli type estimate, Proposition~\ref{lem.caccio}, and
the Sobolev-Poincar\'e type inequality below.  We notice that the Sobolev inequality, more precisely, the Sobolev-Poincar\'e inequality
for the fractional Orlicz-Sobolev space $W^{s,G}(B_r)$ is well known in terms of the Luxemburg norms. However, it does not directly imply
a certain integral version of the Sobolev-Poincar\'e inequality. For the sake of completeness, we need to prove the following
Sobolev-Poincar\'e inequality for functions in $W^{s,G}(B_r)$.

\begin{lemma}\label{lem.sobopoin}(Sobolev-Poincar\'e inequality)
Let $s\in(0,1)$. Then there exists $\theta=\theta(n,s)>1$ such that if
$G$ is an N-function satisfying the $\Delta_2$ condition \eqref{del2} and the $\nabla_2$ condition \eqref{na2} with constants $\kappa$ and $l$,
and $f\in W^{s,G}(B_r)$, then
\begin{align}\label{sobopoin}
\left(\Xint-_{B_r}G\left(\frac{|f-(f)_{B_r}|}{r^s}\right)^{\theta}\,dx\right)^{\frac{1}{\theta}} \le c \Xint-_{B_r}\int_{B_r} G\left(\frac{|f(x)-f(y)|}{|x-y|^s}\right)\,dy dx,
\end{align}
where $c=c(n,s,\kappa, l)>0$.
\end{lemma}
\begin{proof}
We first show that
\begin{equation}\label{fractionalpotential}
|f(x)-(f)_{B_r}| \le c \int_{B_r} \left[\int_{B_r}\frac{ |f(y)-f(z)|}{|y-z|^{n+s}}\,dz\right] \frac{dy}{|x-y|^{n-s}}, \quad \text{a.e. } x\in B_r,
\end{equation}
by using a standard chain argument (see for instance \cite{DLO} and references there in). Fix any Lebesgue's point $x\in B_r$ for $f$. For each $i=0,1,2,\dots$, set $r_i=2^{-i}r$. Then there exists a sequence $\{B^i\}_{i=1}^\infty$ of balls in $B_r$ such that $x\in B^i$, $B^i\subset B_{2r_i}(x)\cap B_r$, $B^{i+1}\subset B^i$, $r_{i} \le  (\text{the radius of }B^i)\le 2r_i$. In particular, we choose $B^0=B_r$ and $B^i=B_{r_i}(x)$ for every $i$ with $r_i \le \mathrm{dist}(x,\partial B_r)$. Then,
\[\begin{split}
|f(x)-(f)_{B_r}| &\le \sum_{i=0}^\infty |(f)_{B^{i+1}}-(f)_{B^i}| \le \sum_{i=0}^\infty \Xint-_{B^{i+1}}|f(y)-(f)_{B^{i}}|\,dy  \\
&\le c \sum_{i=0}^\infty r_i^{-n}\int_{B^{i}}|f(y)-(f)_{B^{i}}|\,dy  \le c \sum_{i=0}^\infty r_i^{-n+s}\int_{B^{i}}\int_{B^{i}}\frac{ |f(y)-f(z)|}{|y-z|^{n+s}}\,dz\,dy.  \\
\end{split}\]
Set
\[
h(y):= \int_{B_r}\frac{ |f(y)-f(z)|}{|y-z|^{n+s}}\,dz, \quad y\in B_r.
\]
Then
\begin{eqnarray*}
|f(x)-(f)_{B_r}| & \leq & c \sum_{i=0}^\infty r_i^{-n+s}\int_{B_{2r_i}(x)\cap B_r}h(y)\,dy \\
&   \leq  & c \sum_{i=0}^\infty \sum_{j=i}^\infty 2^{(n-s)i}r\int_{(B_{2r_{j}}(x)\setminus B_{2r_{j+1}}(x))\cap B_r}h(y)\,dy  \\
&   =  &  c \sum_{j=0}^\infty \left(\sum_{i=0}^j 2^{(n-s)(i-j)}\right)\int_{(B_{2r_{j}}(x)\setminus B_{2r_{j+1}}(x))\cap B_r}r_j^{-n+s} h(y)\,dy  \\
&   \leq  & c \sum_{j=0}^\infty \int_{(B_{2r_{j}}(x)\setminus
B_{2r_{j+1}}(x))\cap B_r}\frac{h(y)}{|x-y|^{n-s}}\,dy \\
&   =  & c \int_{B_r} \frac{h(y)}{|x-y|^{n-s}}\,dy, \end{eqnarray*}
and this is \eqref{fractionalpotential}.

We next prove the desired estimate following the argument in \cite[Thoerem 7]{DE}. To do this, note that for $s>0$ there exists $c(n,s)\ge 1$
such that
\begin{align}
\label{equivalence}
\frac{1}{c(n,s)}\le r^{-s}\int_{B_r}\frac{1}{|x-y|^{n-s}}\,dy  \le c(n,s), \quad \mbox{for every }\ x\in B_r.
\end{align}
Using \eqref{fractionalpotential} with $s$ replaced by $\frac{s}{2}$ and Jensen's inequality, we have
\[\begin{split}
&\Xint-_{B_r} G\left(\frac{|f(x)-(f)_{B_r}|}{r^s}\right)^{\theta}\,dx \\
&\le c \Xint-_{B_r} G\left(r^{-s}\int_{B_r} \left[\int_{B_r}\frac{ |f(y)-f(z)|}{|y-z|^{n+s/2}}\,dz\right] \frac{dy}{|x-y|^{n-s/2}}\right)^{\theta}\,dx\\
&=c \Xint-_{B_r} G\left(r^{-s}\int_{B_r} \left[\int_{B_r}\frac{ |f(y)-f(z)|}{|y-z|^{s}}\,\frac{dz}{|y-z|^{n-s/2}}\right]\, \frac{dy}{|x-y|^{n-s/2}}\right)^{\theta}\,dx\\
&\le c \Xint-_{B_r} \left[r^{-s}\int_{B_r} \int_{B_r}G\left(\frac{ |f(y)-f(z)|}{|y-z|^{s}}\right)\,\frac{dz}{|y-z|^{n-s/2}}\, \frac{dy}{|x-y|^{n-s/2}}\right]^{\theta}\,dx.
\end{split}\]
Denote
\[
L:= \int_{B_r} \int_{B_r}G\left(\frac{ |f(y)-f(z)|}{|y-z|^{s}}\right)\,\frac{dz\, dy}{|y-z|^{n}}.
\]
Then recall the fact that $|y-z|\le 2r$ and use Jensen's inequality and Fubini's theorem, to discover
\[\begin{split}
&\Xint-_{B_r} G\left(\frac{|f(x)-(f)_{B_r}|}{r^s}\right)^{\theta}\,dx \\
&\le c L^\theta \Xint-_{B_r} \left[L^{-1}\int_{B_r} \frac{r^{-s/2}}{|x-y|^{n-s/2}} \left(\int_{B_r}G\left(\frac{ |f(y)-f(z)|}{|y-z|^{s}}\right)\,\frac{dz}{|y-z|^{n}}\right)\,dy\, \right]^{\theta}\,dx\\
&\le c L^{\theta-1} \Xint-_{B_r} \int_{B_r} \left(\frac{r^{-s/2}}{|x-y|^{n-s/2}}\right)^{\theta} \left(\int_{B_r}G\left(\frac{ |f(y)-f(z)|}{|y-z|^{s}}\right)\,\frac{dz}{|y-z|^{n}}\right)\,dy \,dx\\
&= c L^{\theta-1} \Xint-_{B_r} \int_{B_r}\left[ \int_{B_r}\left(\frac{r^{-s/2}}{|x-y|^{n-s/2}}\right)^{\theta} \,dx \right]G\left(\frac{ |f(y)-f(z)|}{|y-z|^{s}}\right)\,\frac{dz\, dy}{|y-z|^{n}}.
\end{split}\]
We now choose $\theta=\theta(n,s)$ such that
\[
1< \theta< \frac{n}{n-s/2}.
\]
From this choice and \eqref{equivalence} with $s$ replaced by $s\theta/2-n(\theta-1)$, we discover
\[
\frac{1}{c}  \le r^{n(\theta-1)-s\theta/2} \int_{B_r} \frac{1}{|x-y|^{(n-s/2)\theta}}\, dx  \le c \quad \mbox{for every }\ x\in B_r.
\]
Consequently,
\[\begin{split}
\Xint-_{B_r} G\left(\frac{|f(x)-(f)_{B_r}|}{r^s}\right)^{\theta}\,dx &\le c (|B_r|^{-1}L)^{\theta-1} \Xint-_{B_r} \int_{B_r} G\left(\frac{ |f(y)-f(z)|}{|y-z|^{s}}\right)\,\frac{dz\, dy}{|y-z|^{n}}\\
&\le c \left( \Xint-_{B_r} \int_{B_r} G\left(\frac{ |f(y)-f(z)|}{|y-z|^{s}}\right)\,\frac{dz\, dy}{|y-z|^{n}}\right)^{\theta}.
\end{split}\]
This finishes the proof.
\end{proof}

\begin{remark}
In Lemma \ref{lem.sobopoin}, we selected $\theta>1$ such that
$\theta \in (1,\frac{n}{n-s/2})$. This selection is not optimal and
it is possible to consider a lager value $\theta$. However the
condition $\theta>1$ is enough in the proof of Theorem
\ref{thm.bounded} below.
\end{remark}

The following technical lemma will be used in the De Giorgi iteration.
\begin{lemma}(\cite[Lemma 7.1]{G})
\label{tech}
Let $\beta >0$ and ${A_{i}}$ be a sequence of real positive numbers such that
\begin{align*}
A_{i+1} \le CB^{i}A_{i}^{1+\beta}
\end{align*}
with $C>0$ and $B>1$.
If $A_{0} \le C^{-\frac{1}{\beta}}B^{{-\frac{1}{\beta^{2}}}}$, then we have
\begin{align*}
A_{i} \le B^{-\frac{i}{\beta}}A_{0} \quad \mbox{hence, in particular, } \ \lim_{i \rightarrow \infty} A_{i}=0.
\end{align*}
\end{lemma}

Now, we are ready to prove the local boundedness of weak solutions to \eqref{pde}.

\begin{theorem}
\label{thm.bounded}
Let $u \in \mathbb{W}^{s,G}(\Omega) \cap L^{g}_{s}(\mr^{n})$ be a weak subsolution to \eqref{pde} and
$B_{r} \Subset \Omega$. Then we have
\begin{align}\label{lb}
\underset{B_{r/2}}{\mathrm{sup}}\, u_{+} \le c_{b} r^{s}G^{-1}\left( \Xint-_{B_r}G\left( \frac{u_{+}}{r^s} \right)dx \right)
+ r^{s}g^{-1}(r^s\mbox{Tail}(u_{+};x_0,r/2)),
\end{align}
where $c_{b}= c_{b}(n,s,p,q,\lambda,\Lambda)>0$.
Moreover, if $u$ is a weak solution to \eqref{pde}, then $u\in L^\infty_{\mathrm{loc}}(\Omega)$ and we have the estimate \eqref{lb1}.
\end{theorem}

\begin{proof}
Suppose that $u$ is a weak subsolution. Fix $B_{r} \Subset \Omega$. For any $j \in \mathbb{N}\cup \{0\}$, write
\begin{align*}
&r_{j}=(1+2^{-j})\frac{r}{2},\ \ \tilde{r}_{j}=\frac{r_{j}+r_{j+1}}{2},\ \ B_{j}=B_{r_{j}}, \ \ \tilde{B}_{j}=B_{\tilde{r}_{j}},\\
&k_{j}=(1-2^{-j})k,\ \ \tilde{k}_{j}=\frac{k_{j}+k_{j+1}}{2},\ \ w_{j}=(u-k_{j})_{+} \ \ \mbox{and}\ \ \tilde{w}_{j}=(u-\tilde{k}_{j})_{+}.
\end{align*}
Note from the above setting that
\begin{align}\label{lb.rel}
B_{j+1} \subset \tilde{B}_{j} \subset B_{j},\quad k_{j} \le \tilde{k}_{j} \le k_{j+1} \quad \mbox{and} \quad w_{j+1} \le \tilde{w}_{j} \le w_{j}.
\end{align}
We take any cut-off functions $\phi_{j} \in C_{0}^{\infty}(\tilde{B}_{j})$ such that $0 \le \phi_{j} \le 1$, $\phi_{j} \equiv 1$ in $B_{j+1}$
and $|D\phi_{j}| \le 2/r_{j+1}$. Putting $\phi_{j}$ into the Caccioppoli inequality \eqref{caccio} with $w_+=\tilde w_j$ (see Remark~\ref{rmk.caccio}) and dividing the inequality by $|B_{j+1}|$, we get
\begin{align}\label{lb.split}
\begin{split}
&\Xint-_{B_{j+1}}\int_{B_{j+1}}G\left(\frac{|\tilde{w}_{j}(x)-\tilde{w}_{j}(y)|}{|x-y|^s}\right)\frac{dxdy}{|x-y|^{n}}\\
&\le c\Xint-_{B_{j}}\int_{B_{j}}G\left(\frac{|\phi_{j}(x)-\phi_{j}(y)|}{|x-y|^s} \max \{ \tilde{w}_{j}(x), \tilde{w}_{j}(y)\} \right)\frac{dxdy}{|x-y|^{n}}\\
&\quad +c\Xint-_{B_{j}}\tilde{w}_{j}(x)\phi_{j}^{q}(x)dx \left( \sup_{y \,\in\, \mathrm{supp}\,\phi_{j}}\int_{\mr^{n} \setminus B_{j}}g\left(\frac{\tilde{w}_{j}(x)}{|x-y|^{s}}\right)\frac{dx}{|x-y|^{n+s}} \right)\\
&=:\rom{1}+\rom{2}.
\end{split}
\end{align}
We first look at the first term $\rom{1}$ in the right-hand side of the above inequality.
Since $|\phi_{j}(x)-\phi_{j}(y)|\le |D\phi_{j}| |x-y|  \le c2^{j}|x-y|/r$, we find
\begin{align}
\label{lb.rom1}
\begin{split}
\rom{1} &\le c\Xint-_{B_{j}}\int_{B_{j}}G\left(2^{j}r^{-1}|x-y|^{1-s} \max \{ \tilde{w}_{j}(x), \tilde{w}_{j}(y) \} \right)\frac{dxdy}{|x-y|^{n}}\\
&\le c2^{qj}\Xint-_{B_{j}}\int_{B_{j}}G\left( \frac{\max \{ \tilde{w}_{j}(x), \tilde{w}_{j}(y)\}}{r^s} \right) \left( \frac{|x-y|}{r} \right)^{(1-s)p} \frac{dxdy}{|x-y|^{n}}\\
&\le c2^{qj}r^{-(1-s)p}\Xint-_{B_{j}}G\left( \frac{\tilde{w}_{j}(x)}{r^s} \right) \left( \int_{B_{j}}\frac{dy}{|x-y|^{n-(1-s)p}} \right) dx\\
&\le c2^{qj}\Xint-_{B_{j}}G\left( \frac{w_{j}(x)}{r^s} \right)dx.
\end{split}
\end{align}
To estimate $\rom{2}$, we write
\begin{align*}
\rom{2}_{1} = \Xint-_{B_{j}}\tilde{w}_{j}(x)\phi_{j}^{q}(x)dx
 \quad\text{and}\quad
\rom{2}_{2} = \sup_{y \, \in\, \mathrm{supp}\,\phi_{j}}\int_{\mr^{n} \setminus B_{j}}g\left(\frac{\tilde{w}_{j}(x)}
{|x-y|^{s}}\right)\frac{dx}{|x-y|^{n+s}}.
\end{align*}
Since $g$ is increasing and $w_{j} \ge \tilde{k}_{j}-k_{j}$ in $\{ u_{j} \ge \tilde{k}_{j} \}$, we have
\begin{align*}
G\left( \frac{w_{j}}{r^s} \right) \ge \frac{1}{q}\frac{w_{j}}{r^s}g\left( \frac{w_{j}}{r^s} \right) \ge \frac{1}{q}\frac{\tilde{w}_{j}}{r^s}g\left( \frac{\tilde{k}_{j}-k_j}{r^s} \right) \ge c2^{-(q-1)j}\frac{\tilde{w}_{j}}{r^s}g\left( \frac{k}{r^s} \right).
\end{align*}
Thus
\begin{align}\label{loc.rom2.c1}
\rom{2}_{1} \le c2^{(q-1)j}\frac{r^s}{g\left( k/{r^s} \right)}\Xint-_{B_{j}}G\left( \frac{w_{j}}{r^s} \right)dx.
\end{align}
In order to estimate $\rom{2}_{2}$, we notice that for $x \in \mr^{n} \setminus B_{j}$ and $y \in \tilde{B}_{j}$,
\begin{align*}
\frac{|x-x_0|}{|x-y|} \le \frac{|x-y|+|y-x_0|}{|x-y|} \le 1+\frac{\tilde{r}_{j}}{r_{j}-\tilde{r}_{j}} \le 2^{j+4}.
\end{align*}
This and \eqref{lb.rel} imply
\begin{align}\label{loc.rom2.c2}\begin{split}
\rom{2}_2&\le \sup_{y\in \tilde B_j}\int_{\mr^{n} \setminus B_{r/2}}g\left(\frac{w_0}{|x-y|^{s}}\right)\frac{dx}{|x-y|^{n+s}}\\
&\le c2^{(n+sq)j}\int_{\mr^{n} \setminus B_{r/2}}g\left(\frac{u_+}{|x-x_0|^{s}}\right)\frac{dx}{|x-x_0|^{n+s}}\\
&=c2^{(n+sq)j}\mbox{Tail}(u_{+};x_0,r/2).
\end{split}\end{align}
In light of \eqref{loc.rom2.c1} and \eqref{loc.rom2.c2}, we thus deduce
\begin{align}\label{lb.rom2}
\rom{2} \le c2^{(n+sq+q)j}\frac{r^s}{g\left( k/{r^s} \right)}\left( \Xint-_{B_{j}}G\left( \frac{w_{j}}{r^s} \right)dx \right)\mbox{Tail}(u_{+};x_0,r/2).
\end{align}
Combining \eqref{lb.split}, \eqref{lb.rom1}, and \eqref{lb.rom2}, and applying the Sobolev-Poincar\'{e} inequality \eqref{sobopoin} to the left-hand side of \eqref{lb.split}, we have
\begin{align}\label{lb.in1}
\begin{split}
&\left( \Xint-_{B_{j+1}}G^{\theta}\left(\frac{|\tilde{w}_{j}-(\tilde{w})_{B_{j+1}}|}{r_{j+1}^{s}}\right)dx \right)^{\frac{1}{\theta}}\\
&\le c2^{(n+sq+q)j} \left[ \Xint-_{B_{j}}G\left( \frac{w_{j}}{r^s} \right)dx + \frac{r^s}{g\left( k/{r^s} \right)}
\left( \Xint-_{B_{j}}G\left( \frac{w_{j}}{r^s} \right)dx \right)\mbox{Tail}(u_{+};x_0,r/2) \right]
\end{split}
\end{align}
for some $\theta=\theta(n,s) > 1$. On the other hand, recalling the definition of $r_{j+1}$ and using Jensen's inequality and \eqref{lb.rel}, we discover
\begin{align}
\label{lb.in2}
\begin{split}
\left( \Xint-_{B_{j+1}}G^{\theta}\left( \frac{\tilde{w}_{j}}{r^{s}} \right)dx \right)^{\frac{1}{\theta}}
&\le c\left( \Xint-_{B_{j+1}}G^{\theta}\left( \frac{|\tilde{w}_{j} - (\tilde{w}_{j})_{B_{j+1}}|}{r^{s}}\right)dx \right)^{\frac{1}{\theta}} + cG\left( \frac{(\tilde{w}_{j})_{B_{j+1}}}{r^{s}} \right)\\
&\le c\left( \Xint-_{B_{j+1}}G^{\theta}\left( \frac{|\tilde{w}_{j} - (\tilde{w}_{j})_{B_{j+1}}|}{r_{j+1}^{s}}\right)dx \right)^{\frac{1}{\theta}} + c
\Xint-_{B_{j}}G\left( \frac{w_{j}}{r^{s}} \right)dx.
\end{split}
\end{align}

Let us estimate the left-hand side of \eqref{lb.in2}. Notice that the relations in \eqref{lb.rel} yield
\begin{align*}
&G^{\theta}\left( \frac{\tilde{w}_{j}}{r^{s}} \right)  \ge G^{\theta-1}\left( \frac{\tilde{w}_{j}}{r^{s}} \right)G\left( \frac{w_{j+1}}{r^{s}} \right) \ge G^{\theta-1}\left( \frac{k_{j+1}-\tilde{k}_{j}}{r^{s}} \right)G\left( \frac{w_{j+1}}{r^{s}} \right).
\end{align*}
Therefore it follows that
\begin{align}\label{lb.in3}
\begin{split}
G^{\frac{\theta-1}{\theta}}\left( \frac{k}{r^{s}} \right)\left( \Xint-_{B_{j+1}}G\left( \frac{w_{j+1}}{r^{s}} \right)dx \right)^{\frac{1}{\theta}} &\le c2^{qj}G^{\frac{\theta-1}{\theta}}\left( \frac{k_{j+1}-\tilde{k}_{j}}{r^{s}} \right)\left( \Xint-_{B_{j+1}}G\left( \frac{w_{j+1}}{r^{s}} \right)dx \right)^{\frac{1}{\theta}}\\
&\le c2^{qj}\left( \Xint-_{B_{j+1}}G^{\theta}\left( \frac{\tilde{w}_{j}}{r^{s}} \right)dx \right)^{\frac{1}{\theta}}.
\end{split}
\end{align}
Taking into account \eqref{lb.in1}, \eqref{lb.in2} and \eqref{lb.in3}, we deduce that
\begin{align}\label{lb.it1}
\begin{split}
&G^{\frac{\theta-1}{\theta}}\left( \frac{k}{r^{s}} \right)\left( \Xint-_{B_{j+1}}G\left( \frac{w_{j+1}}{r^{s}} \right)dx \right)^{\frac{1}{\theta}}\\
&\le c2^{(n+sq+2q)j} \left[ \Xint-_{B_{j}}G\left( \frac{w_{j}}{r^s} \right)dx + \frac{r^s}{g\left( k/{r^s} \right)}\left( \Xint-_{B_{j}}G\left( \frac{w_{j}}{r^s} \right)dx \right)\mbox{Tail}(u_{+};x_0,r/2) \right].
\end{split}
\end{align}
Denote
\begin{align*}
a_j:=\frac{1}{G(k/r^s)}\Xint-_{B_j}G\left( \frac{w_j}{r^s} \right)dx.
\end{align*}
Then \eqref{lb.it1} is identical to
\begin{align*}
a_{j+1} \le c_{2}2^{(n+sq+2q)\theta j}\left[ 1 + \frac{r^s}{g\left( k/{r^s} \right)}\mbox{Tail}(u_{+};x_0,r/2) \right]^{\theta}a_{j}^{\theta}
\end{align*}
for some $c_{2}>0$ depending on $n,s,p,q,\lambda$ and $\Lambda$. At this stage, choose
\begin{align*}
k = r^{s}G^{-1}\left( c_{3}\Xint-_{B_r}G\left( \frac{u_{+}}{r^s} \right)dx \right) + r^{s}g^{-1}(r^s\mbox{Tail}(u_{+};x_0,r/2)),
\end{align*}
where $c_{3} = (c_{2}2^{\theta})^{\frac{1}{\theta - 1}}2^{\frac{(n+sq+2q)\theta}{(\theta - 1)^{2}}}$.
Then we see that
\begin{align*}
a_{j+1} \le (c_{2}2^{\theta})2^{(n+sq+2q)\theta j}a_{j}^{\theta} \quad \mbox{and} \quad a_{0} \le c_3^{-1}= (c_{2}2^{\theta})^{-\frac{1}{\theta - 1}}2^{-\frac{(n+sq+2q)\theta}{(\theta - 1)^{2}}}.
\end{align*}
Set $c_{b} = \max \left\{ c_{3}^{1/p},c_{3}^{1/q} \right\}$. Since lemma \ref{tech} implies $a_{j} \rightarrow 0$ as $j \rightarrow \infty$, we discover
\begin{align*}
\underset{B_{r/2}}{\mathrm{sup}}\, u_{+} \le k \le c_{b}r^{s}G^{-1}\left( \Xint-_{B_r}G\left( \frac{u_{+}}{r^s} \right)dx \right) + r^{s}g^{-1}(r^s\mbox{Tail}(u_{+};x_0,r/2)),
\end{align*}
and this is \eqref{lb}.\\
If $u$ is a weak solution, then $-u$ is a weak subsolution. Then we have the estimate \eqref{lb} with $u_{+}$ replaced by 
$(-u)_{+}=u_{-}$. This completes the proof.
\end{proof}

\subsection{H\"older continuity}
We complete the proof of Theorem~\ref{mainthm} by obtaining \eqref{holder}.
Let $u \in \mathbb{W}^{s,G}(\Omega) \cap L^{g}_{s}(\mr^{n})$ be a weak solution to \eqref{pde}.
Let $B_{r} \equiv B_{r}(x_0) \Subset \Omega$. For  $\alpha\in(0,1)$, $\sigma\in(0,1)$ and $i=0,1,2,\dots$, we write
\begin{align}
\label{ri}
r_i:=\sigma^{i}\frac{r}{2} \quad \text{and} \quad B_i=B_{r_i}(x_0)
\end{align}
and define
\begin{align}
\label{ori}
\omega(r_i):=\left( \frac{r_i}{r_0} \right)^{\alpha}\omega(r_0)=\sigma^{\alpha i}\omega(r_0)
\end{align}
with
\begin{align}\label{omega0}
\omega(r_0):=2\left( c_br^{s}G^{-1}\left( \Xint-_{B_r}G\left( \frac{|u|}{r^s} \right)dx \right) +r^{s}g^{-1}(r^s\mbox{Tail}(u;x_0,r/2))  \right),
\end{align}
where $c_{b}$ is as in \eqref{lb1}.

For the proof of \eqref{holder}, it is enough to show the following
oscillation  decay estimate.
\begin{lemma}
\label{lem.osc}
Under the above setting, there exist small  $\alpha,\sigma\in (0,1)$ depending on $n,s,p,q,\lambda$ and $\Lambda$ such that
for every $i \geq 0$,
\begin{align}
\label{osc}
\underset{B_{i}}{\mathrm{osc}} \, u := \sup_{B_{i}}u - \inf_{B_{i}}u \le \omega(r_{i}).
\end{align}
\end{lemma}

\begin{proof}
First of all, we assume that
\begin{align}\label{sigma1}
\alpha\le \frac{sp}{2(p-1)} \quad \text{and}\quad \sigma<\frac{1}{4}.
\end{align}
We prove this lemma by induction. Obviously, \eqref{osc} holds true for $i=0$ from \eqref{lb1} and the definition of $\omega(r_0)$.
Suppose that for some $j \geq 0$,
\begin{align}\label{ho.hypo}
\underset{B_i}{\mathrm{osc}} \, u \le \omega(r_i) \quad  \mbox{for all }\ i \in \{ 0,1,2, \cdots, j \},
\end{align}
and then we will prove \eqref{osc} for $i=j+1$.
We define $u_{j}$ by
\begin{align*}
u_{j}:=
\begin{cases}\displaystyle
u-\inf_{B_j}u, &\quad \displaystyle \mbox{if } \ |2B_{j+1} \cap \{ u \ge \inf_{B_j}u + \omega (r_j)/2 \}| \ge \tfrac{1}{2}|2B_{j+1}|,\\
\displaystyle\omega(r_{j})-(u-\inf_{B_j}u), &\quad\displaystyle \mbox{if } \ |2B_{j+1} \cap \{ u \le \inf_{B_j}u + \omega (r_j)/2 \}| \ge \tfrac{1}{2}|2B_{j+1}|,\\
\end{cases}
\end{align*}
where $2B_{j+1}:=B_{2r_{j+1}}$.
Then $u_{j} \ge 0$ in $B_{j}$ and
\begin{align}\label{density}
\frac{|2B_{j+1} \cap \{ u_j \ge \omega (r_j)/2 \}|}{|2B_{j+1}|} \ge \frac{1}{2}.
\end{align}
We divide the remaining part of the proof into three steps.

\smallskip

\textit{Step 1.} We first estimate Tail$(u_{j};x_0,r_{j})$. Define $T_{1}$ and $T_{2}$ as follows:
\begin{align}\label{ho.tail.sp}
\begin{split}
\mbox{Tail}(u_j;x_0,r_j)&=\sum_{i=1}^{j} \int_{B_{i-1} \setminus B_i}g\left( \frac{|u_{j}(x)|}{|x-x_0|^s} \right)\frac{dx}{|x-x_0|^{n+s}}\\
&\qquad +\int_{\mr^{n} \setminus B_0} g\left( \frac{|u_{j}(x)|}{|x-x_0|^s} \right)\frac{dx}{|x-x_0|^{n+s}}\\
&=:T_1+T_2.
\end{split}
\end{align}
Before estimating $T_{1}$ and $T_{2}$, observe that the definition of $u_{j}$ and the induction hypothesis \eqref{ho.hypo} imply
\begin{align}\label{ho.tail.u1}
\sup_{B_i}|u_j| \le 2\omega(r_i) \quad \mbox{for all } \ i \le j.
\end{align}
Moreover, the local boundedness of $u$ implies
\begin{align}\label{ho.tail.u2}
|u_{j}| \le |u|+\omega(r_{j})+\sup_{B_{j}}|u| \le |u|+2\omega(r_{0}).
\end{align}
We now estimate $T_{1}$. Recall \eqref{ho.tail.u1} to find
\begin{align}\label{ho.t1}
\begin{split}
T_1 &\le \sum_{i=1}^{j} \int_{B_{i-1} \setminus B_i}g\left( \frac{\sup_{B_{i-1}}|u_j|}{|x-x_0|^s} \right)\frac{dx}{|x-x_0|^{n+s}}\\
&\le c\sum_{i=1}^{j} \int_{B_{i-1} \setminus B_i}g\left( \frac{\omega(r_{i-1})}{r_{i}^s} \right)\left( \frac{r_{i}^s}{|x-x_0|^s} \right)^{p-1}\frac{dx}{|x-x_0|^{n+s}}\\
&= c\sum_{i=1}^{j} r_{i}^{s(p-1)}g\left( \frac{\omega(r_{i-1})}{r_{i}^s} \right)\int_{B_{i-1} \setminus B_i}\frac{dx}{|x-x_0|^{n+sp}} \le c\sum_{i=1}^{j} \frac{1}{r_{i}^{s}}g\left( \frac{\omega(r_{i-1})}{r_{i}^{s}} \right).
\end{split}
\end{align}

In order to estimate $T_{2}$, we write $h(t):=\tilde
g(t^{1/(q-1)}))$ with $\tilde g(t):=G(t)/t$. Then by \eqref{growth},
we have $p\tilde g(t)\le g(t) \le q \tilde g(t)$, $h(t)$ is increasing
and $h(t)/t$ is decreasing for $t$. Then there exists a concave
function $\psi$ such that $\frac{1}{2}\psi(t)\leq h(t)\leq \psi(t)$,
see \cite[Lemma 2.2]{Ok18}. Considering \eqref{ho.tail.u2} and using
Jensen's inequality with respect to the measure
$\frac{dx}{|x-x_{0}|^{n+s}}$, we find
\begin{align*}%\label{ho.t2.in1}
\begin{split}
T_2 &\le c\int_{\mr^{n} \setminus B_{0}} \tilde g\left( \frac{\omega(r_{0})}{|x-x_0|^s} \right)\frac{dx}{|x-x_0|^{n+s}}
+ c\int_{\mr^{n} \setminus B_0} g\left( \frac{|u(x)|}{|x-x_0|^s} \right)\frac{dx}{|x-x_0|^{n+s}}\\
&= c\int_{\mr^{n} \setminus B_{0}} h \left( \left(\frac{\omega(r_{0})}{|x-x_0|^s}\right)^{q-1} \right)\frac{dx}{|x-x_0|^{n+s}}  + c\mbox{Tail}(u;x_{0},r_0)\\
&\le \frac{c}{r_0^{s}} h \left(r_0^{s}\int_{\mr^{n} \setminus B_{0}} \left( \frac{\omega(r_{0})}{|x-x_0|^s} \right)^{q-1}\frac{dx}{|x-x_0|^{n+s}} \right) + c\mbox{Tail}(u;x_{0},r_0)\\
&\le \frac{c}{r_0^{s}} h \left( \left( \frac{\omega(r_{0})}{r_{0}^{s}} \right)^{q-1} \right) + c\mbox{Tail}(u;x_{0},r_0)\\
&\le \frac{c}{r_{0}^{s}} g \left( \frac{\omega(r_{0})}{r_{0}^{s}} \right) + c\mbox{Tail}(u;x_{0},r_0).
\end{split}
\end{align*}
We recall \eqref{omega0} to discover
\begin{align}\label{ho.t2}
T_{2} \le \frac{c}{r_{0}^{s}}g \left( \frac{\omega(r_{0})}{r_{0}^{s}} \right) \le \frac{c}{r_{1}^{s}}g \left( \frac{\omega(r_{0})}{r_{1}^{s}} \right).
\end{align}
Then we combine \eqref{ho.tail.sp}, \eqref{ho.t1}, and \eqref{ho.t2}, and recall \eqref{ri} and \eqref{ori} to have
\begin{align}
\label{ho.tail}
\begin{split}
\mbox{Tail}(u_j;x_0,r_j) &\le \sum_{i=1}^{j} \frac{c}{r_{i}^{s}}g\left( \frac{\omega(r_{i-1})}{r_{i}^{s}} \right) = \sum_{i=1}^{j} \frac{c}{r_{i}^{s}}g\left( \frac{\omega(r_{j})}{r_{j+1}^{s}}\sigma^{(s-\alpha)(j-i+1)} \right)\\
&\le \frac{c}{r_{j+1}^s}g\left( \frac{\omega(r_{j})}{r_{j+1}^{s}} \right)\sum_{i=1}^{j}\sigma^{(sp-\alpha(p-1))(j-i+1)}\\
&\le \frac{c}{r_{j+1}^s}g\left( \frac{\omega(r_{j})}{r_{j+1}^{s}} \right)\frac{\sigma^{sp-\alpha(p-1)}}{1-\sigma^{sp-\alpha(p-1)}}\\
&\le \frac{c}{r_{j+1}^s}g\left( \frac{\omega(r_{j})}{r_{j+1}^{s}} \right)\sigma^{sp-\alpha(p-1)},
\end{split}
\end{align}
by taking $\sigma>0$ sufficiently small so that
\begin{align}\label{sigma2}
\sigma^{sp-\alpha(p-1)}\le  \sigma^{\frac{sp}{2}} \le  \frac{1}{2}.
\end{align}

\smallskip
\textit{Step 2.} In this step, we look at
\begin{align}\label{epsilon}
\frac{|2B_{j+1}\cap\{ u_{j} \le 2\ep\omega(r_j) \}|}{|2B_{j+1}|}, \quad \mbox{where }\ \ep:=\sigma^{\frac{sp-\alpha(p-1)}{q-1}}\le \sigma^{\frac{sp}{2(q-1)}} <1.
\end{align}
For $k>0$ to be determined later, we write
\begin{align*}
v:=\min \left\{\left[ \log \left( \frac{\omega(r_j)/2+\ep\omega(r_j)}{u_j+\ep\omega(r_j)} \right) \right]_{+}, k \right\}.
\end{align*}
Then applying Lemma~\ref{lo} with $u=u_j$, $r=2r_{j+1}$, $R=r_j$, $a \equiv \omega(r_{j})$, $b \equiv \exp{(k)}$
and $d=\ep\omega(r_j)$ and using \eqref{ho.tail}, we find
\begin{align}
\label{vestimate}
\Xint-_{2B_{j+1}}|v-(v)_{2B_{j+1}}|\,dx \le c\left[ 1 + \frac{g(\omega(r_{j})/r_{j+1}^{s})}{g(\ep\omega(r_j)/r_{j+1}^{s})}\sigma^{sp-\alpha(p-1)} \right] \le c.
\end{align}
On the other hand, from \eqref{density} we see
\begin{align*}
k
&= \frac{1}{|2B_{j+1}\cap\{ u_{j} \ge \omega(r_j)/2 \}|}\int_{2B_{j+1}\cap\{ v=0 \}}k \, dx\le \frac{2}{|2B_{j+1}|}\int_{2B_{j+1}}(k-v)\, dx = 2[k-(v)_{2B_{j+1}}].
\end{align*}
Integrating the above inequality over $2B_{j+1}\cap\{ v=k \}$ and using \eqref{vestimate},
we get
\begin{align*}
\frac{|2B_{j+1} \cap \{ v=k \}|}{|2B_{j+1}|}k &\le \frac{2}{|2B_{j+1}|}\int_{2B_{j+1} \cap \{ v=k \}}(k-(v)_{2B_{j+1}}) \, dx\\
& \le  \frac{2}{|2B_{j+1}|} \int_{2B_{j+1}}|v-(v)_{2B_{j+1}}|\,dx \le c.
\end{align*}
Here we assume $\sigma>0$ is sufficiently small so that
\begin{align}\label{sigma3}
\sqrt{\epsilon}= \sigma^{\frac{sp-\alpha(p-1)}{2(q-1)}}\le  \sigma^{\frac{sp}{4(q-1)}}  \le \frac{1}{6},
\end{align}
and take
\begin{align*}
k = \log{\left( \frac{\omega(r_{j})/2 + \ep \omega(r_{j})}{3\ep \omega(r_{j})} \right)} \ge \log\left(\frac{1}{6\epsilon}\right)\ge \frac{1}{2}\log\left(\frac1 \epsilon\right),
\end{align*}
from which, together with \eqref{epsilon}, we discover
\begin{align*}
\frac{|2B_{j+1}\cap\{ u_{j} \le 2\ep\omega(r_j) \}|}{|2B_{j+1}|} \le \frac{c}{k} \le \frac{c_{4}}{\log{(1/\sigma)}}
\end{align*}
for some $c_{4}>0$ depending on $n,p,q,\lambda$ and $\Lambda$.

\smallskip
\textit{Step 3.} Finally, we prove \eqref{osc} for $i=j+1$. For any $m = 0,1,2,\dots$, we write
\begin{align*}
&\rho_{m}=(1+2^{-m})r_{j+1},\ \  \tilde{\rho}_{m}=\frac{r_{m}+r_{m+1}}{2},\ \ B^{m}=B_{\rho_{m}},\ \ \tilde{B}^{m}=B_{\tilde{\rho}_{m}},\\
&k_{m}=(1+2^{-m})\ep\omega(r_{j})\ \ \mbox{and}\ \ w_{m}=(k_{m}-u_{j})_{+}=(u_{j}-k_m)_-.
\end{align*}
Note that $r_{j+1}<\rho_m\le 2r_{j+1}$ and $\ep\omega(r_{j})< k_m \le 2\ep\omega(r_{j})$. Take cut-off functions
$\phi_{m} \in C_{0}^{\infty}(\tilde{B}^{m})$ such that $0 \le \phi_{m} \le 1$, $\phi_{m} \equiv 1$ in $B^{m+1}$ and $|D\phi_{m}|<c/\rho_{m}$. Applying the Caccioppoli inequality \eqref{caccio} to $w_-=w_m$, $\phi=\phi_{m}$ and $B_r=B^m$, we have
\begin{align}\label{ho.ca}
\begin{split}
&\int_{B^{m+1}}\int_{B^{m+1}}G\left(\frac{|w_{m}(x)-w_{m}(y)|}{|x-y|^s}\right)\frac{dxdy}{|x-y|^{n}}\\
&\quad\le c\int_{B^{m}}\int_{B^{m}}G\left(\frac{|\phi_{m}(x)-\phi_{m}(y)|}{|x-y|^s} \max \{w_{m}(x),w_{m}(y)\} \right)\frac{dxdy}{|x-y|^{n}}\\
&\qquad +c\int_{B^{m}}w_{m}(x)\phi_{m}^{q}(x)dx\left( \sup_{y \in \tilde{B}^{m}}\int_{\mr^{n} \setminus {B^{m}}}
g\left(\frac{w_{m}(x)}{|x-y|^{s}}\right)\frac{dx}{|x-y|^{n+s}}\right).
\end{split}
\end{align}
As in the proof of local boundedness, we use the Sobolev-Poincar\'{e} inequality \eqref{sobopoin},  Jensen's inequality and \eqref{ho.ca}, to find
\begin{align}
\label{ho.sp}
\begin{split}
\rom{1} &:= \left( \Xint-_{B^{m+1}}G^{\theta}\left( \frac{w_{m}}{\rho_{m+1}^s} \right)\,dx \right)^{\frac{1}{\theta}}\\
&\le c\left( \Xint-_{B^{m+1}}G^{\theta}\left( \frac{w_{m}-(w_{m})_{B^{m+1}}}{\rho_{m+1}^s} \right)\,dx \right)^{\frac{1}{\theta}} + c\left( \Xint-_{B^{m+1}}G^{\theta}\left( \frac{(w_{m})_{B^{m+1}}}{\rho_{m+1}^s} \right)\,dx \right)^{\frac{1}{\theta}}\\
&\le c\Xint-_{B^{m+1}}\int_{B^{m+1}}G\left(\frac{|w_{m}(x)-w_{m}(y)|}{|x-y|^{s}}\right)\,\frac{dxdy}{|x-y|^{n}} + cG\left( \frac{(w_{m})_{B^{m+1}}}{\rho_{m+1}^s} \right)\\
&\le c\Xint-_{B^{m}}\int_{B^{m}}G\left(\frac{|\phi_{m}(x)-\phi_{m}(y)|}{|x-y|^{s}} \max \{w_{m}(x),w_{m}(y)\} \right)\,\frac{dxdy}{|x-y|^{n}}\\
&\quad +c\Xint-_{B^{m}}w_{m}(x)\phi_{m}^{q}(x)\,dx\left( \sup_{y \in \tilde{B}^{m}}\int_{\mr^{n} \setminus {B^m}}g\left(\frac{w_{m}(x)}{|x-y|^{s}}\right)\,\frac{dx}{|x-y|^{n+s}}\right)\\
&\quad +c\Xint-_{B^{m+1}}G\left( \frac{w_{m}}{\rho_{m+1}^{s}} \right)\,dx\\
&=: \rom{2} + \rom{3} + \rom{4}.
\end{split}
\end{align}
We write
\begin{align*}
A_{m} := \frac{|B^{m}\cap\{ u_{j} \le k_{m} \}|}{|B^{m}|}.
\end{align*} From the definition of $u_{j}$, $k_{m}$ and $A_{m}$, we
estimate $\rom{1}$ as follows:
\begin{align}\label{ho.rom1}
\begin{split}
\rom{1} &\ge \frac{1}{|B^{m+1}|^{\frac{1}{\theta}}} \left( \int_{B^{m+1}\cap\{ u_j \le k_{m+1} \}}G^{\theta}\left(\frac{k_{m}-k_{m+1}}{\rho_{m+1}^s}\right)\, dx \right)^{\frac{1}{\theta}} = A_{m+1}^{\frac{1}{\theta}}G\left(\frac{k_{m}-k_{m+1}}{\rho_{m+1}^s}\right)\\
&\ge c 2^{-qm}A_{m+1}^{\frac{1}{\theta}}G\left(\frac{\ep \omega(r_j)}{r_{j+1}^s}\right).
\end{split}
\end{align}
Since $|\phi_{m}(x) - \phi_{m}(y)|\le c |x-y|/r_{j+1}$, we find
\begin{align}\label{ho.rom2}
\begin{split}
\rom{2}
&\le c2^{qm}r_{j+1}^{-n}\int_{B^{m}}\int_{B^{m}}G\left( \frac{\max \{w_{m}(x),w_{m}(y)\}}{r_{j+1}^s} \right) \left( \frac{|x-y|}{r_{j+1}} \right)^{(1-s)p}\, \frac{dxdy}{|x-y|^{n}}\\
&\le c2^{qm}r_{j+1}^{-n-(1-s)p}\int_{B^{m} \cap \{ u_{j} \le k_m \}}\int_{B^{m}}G\left( \frac{k_m}{r_{j+1}^s} \right) \,\frac{dydx}{|x-y|^{n-(1-s)p}}\\
&\le c2^{qm}r_{j+1}^{-n}G\left( \frac{\ep\omega(r_j)}{r_{j+1}^s} \right)|B^{m} \cap \{ u_{j} \le k_{m} \}|
\le c2^{qm}G\left( \frac{\ep\omega(r_j)}{r_{j+1}^s} \right)A_{m}.
\end{split}
\end{align}
As for $\rom{3}$,
set
\begin{align*}
\rom{3}_{1} = \Xint-_{B^{m}}w_{m}(x)\phi^{q}(x)\, dx
\quad \mbox{and}\quad
\rom{3}_{2} = \sup_{y \in \tilde{B}^{m}}\int_{\mr^{n} \setminus {B^{m}}}g\left(\frac{w_{m}(x)}{|x-y|^{s}}\right)\, \frac{dx}{|x-y|^{n+s}}.
\end{align*}
Then we have
\begin{align*}
\rom{3}_{1} \le |B^{m}|^{-1}\int_{B^{m} \cap \{ u_{j} \le k_{m} \}}k_{m}dx \le cr_{j+1}^{-n}[\ep\omega(r_j)]|B^{m} \cap \{ u_{j} \le k_{m} \}|.
\end{align*}
Using the fact that
$\frac{|x-x_{0}|}{|x-y|} \le 1 + \frac{|y-x_{0}|}{|x-y|} \le 1 + \frac{\tilde{\rho}_{m}}{\rho_{m} - \tilde{\rho}_{m}} \le c  2^{m}
$
for $x \in \mr^{n} \setminus B^{m}$ and  $y \in \tilde{B}^{m}$,
we have
\begin{align*}
\rom{3}_{2} \le \int_{\mr^{n} \setminus B_{j+1}}g\left(\frac{w_{m}(x)}{|x-x_0|^{s}}2^{sm}\right)2^{(n+s)m}\frac{dx}{|x-x_0|^{n+s}}
\le c2^{(n+sq)m}\mbox{Tail}(w_{m};x_0,r_{j+1}).
\end{align*}
Moreover, since $u_j\ge 0$ in $B_j$, we have $w_m \le k_m \le 2 \ep \omega(r_j)$ in $B_j$ and $w_m\le k_m+|u_j|
\le  2 \ep \omega(r_j)+|u_j|$ in $\R^n\setminus B_j$. Then from \eqref{ho.tail}, we see
\begin{align*}
&\mbox{Tail}(w_{m};x_0,r_{j+1})
\le c\int_{\mr^{n} \setminus B_{j+1}}g\left( \frac{\ep \omega(r_j)}{|x-x_0|^s} \right)\frac{dx}{|x-x_0|^{n+s}} + c\mbox{Tail}(u_{j};x_0,r_{j})\\
&\le c\int_{\mr^{n} \setminus B_{j+1}}g\left( \frac{\ep\omega(r_{j})}{r_{j+1}^{s}} \right) \left(\frac{r_{j+1}}{|x-x_0|}\right)^{(p-1)s} \, \frac{dx}{|x-x_0|^{n+s}} + \frac{c}{r_{j+1}^s}g\left( \frac{\omega(r_{j})}{r_{j+1}^{s}} \right)\sigma^{sp-\alpha(p-1)}\\
&\le \frac{c}{r_{j+1}^s}g\left( \frac{\ep\omega(r_{j})}{r_{j+1}^{s}} \right).
\end{align*}
Therefore we obtain
\begin{align}\label{ho.rom3}
\rom{3} \le c 2^{(n+sq)m} r_{j+1}^{-n}|B^{m} \cap \{ u_{j} \le k_{m} \}|  \frac{\ep\omega(r_j)}{r_{j+1}^s}g\left( \frac{\ep\omega(r_j)}{r_{j+1}^s} \right) \le 2^{(n+sq)m} G\left( \frac{\ep\omega(r_j)}{r_{j+1}^s} \right)A_{m}.
\end{align}
We recall the notation for $\rom{4}$ to find
\begin{align}\label{ho.rom4}
\rom{4} \le c G\left( \frac{\ep\omega(r_j)}{r_{j+1}^{s}} \right)A_{m}.
\end{align}
We finally combine \eqref{ho.sp},\eqref{ho.rom1},\eqref{ho.rom2},\eqref{ho.rom3}, and \eqref{ho.rom4}, to discover
\begin{align*}
A_{m+1} \le c2^{(n+sq+2q)\theta m}A_{m}^{1+\beta}, \quad \mbox{where }\ \beta=\theta -1.
\end{align*}
Recall that
\begin{align*}
A_0=\frac{|2B_{j+1} \cap \{ u_j \le 2\ep\omega(r_j) \}|}{|2B_{j+1}|}
\le \frac{c_{4}}{\log(1/\sigma)}
\end{align*}
and choose $\sigma>0$ sufficiently large such that
\begin{align}\label{sigma3}
\frac{c_{4}}{\log(1/\sigma)} \le c^{-1/\beta}2^{-[n+sq+2q]\theta/\beta^{2}}.
\end{align}
Here, we notice that the constant $\sigma$ is determined from \eqref{sigma1}, \eqref{sigma2}, and \eqref{sigma3}, hence depends only on $n,s,p,q,\lambda$ and $\Lambda$.
Then we apply Lemma~\ref{tech} to see that $\lim_{m\to
\infty}A_{m}=0$, which implies
\begin{align*}
u_j > \epsilon \omega(r_j)  \quad \text{in }\ B_{j+1},
\end{align*}
and recall our notations for $u_j$, $\omega(r_{j+1})$ and
$\epsilon$, to conclude
\begin{align*}
\underset{B_{j+1}}{\mathrm{osc}}\, u \le (1-\epsilon)\omega(r_j) =
\frac{1-\sigma^{\frac{sp-\alpha(p-1)}{q-1}}}{\sigma^{\alpha}}\omega(r_{j+1})
\le \frac{1-\sigma^{\frac{sp}{q-1}}}{\sigma^{\alpha}}\omega(r_{j+1})
\le \omega(r_{j+1}),
\end{align*}
by taking $\alpha=\alpha(n,s,p,q,\lambda,\Lambda)>0$ sufficiently small so that
$$
\sigma^{\alpha} \ge 1-\sigma^{\frac{sp}{q-1}}.
$$
This completes the proof.
\end{proof}

\section{Existence and Uniqueness of Weak Solution}
\label{sec5}

In this section we prove the existence and uniqueness of weak
solution to \eqref{pde} with a Dirichlet boundary condition. The
proof is based on a direct method in the calculus of variations. We
refer the reader to \cite{G,DKP1,FS} for the details. Before
introducing the compact embedding in $W^{s,G}(\Omega)$, we need the
following definition.

\begin{definition}
Let $A$ and $B$ be two Young functions.
We say B grows essentially more slowly near infinity than $A$ if
\begin{align}\label{ess}
\lim_{t \rightarrow \infty} \frac{B(\lambda t)}{A(t)} = 0
\end{align}
for every $\lambda > 0$.
\end{definition}
Note that the condition \eqref{ess} is equivalent to
\begin{align*}
\lim_{t \rightarrow \infty} \frac{A^{-1}(t)}{B^{-1}(t)} = 0.
\end{align*}
Let $s \in (0,1)$ and let $A$ be a Young function such that
\begin{align}\label{target}
\int_{0}^{1} \left( \frac{t}{A(t)} \right)^{\frac{s}{n-s}} dt < \infty
\quad\mbox{and}\quad
\int_{1}^{\infty} \left( \frac{t}{A(t)} \right)^{\frac{s}{n-s}} dt = \infty.
\end{align}
Then $A_{\frac{n}{s}}$ is given by
\begin{align}\label{t.space}
A_{\frac{n}{s}}(t) := A(H^{-1}(t)) \quad  \mbox{for }\  t \ge 0,
\end{align}
where the function $H:[0,\infty) \rightarrow [0,\infty)$ obeys
\begin{align*}
H(t) := \left( \int_{0}^{t} \left( \frac{\tau}{A(\tau)} \right)^{\frac{s}{n-s}}d\tau\right)^{\frac{n-s}{n}} \quad  \mbox{for }\ t \ge 0.
\end{align*}

\begin{lemma}(\cite[Theorem 3.5]{ACPS})
(Compact embedding) \label{cpt} Let $s \in (0,1)$ and let $A$ be a
Young function fulfilling \eqref{target}. Let $A_{\frac{n}{s}}$ be
the Young function defined as in \eqref{t.space}. Assume that $B$ is
a Young function. Then the following properties are equivalent.
\begin{enumerate}
\item B grows essentially more slowly near infinity than $A_{\frac{n}{s}}$.
\item The embedding
\begin{align*}
W^{s,A}(U) \rightarrow L^{B}(U)
\end{align*}
is compact for every bounded domain $U \subset \mr^{n}$ with Lipschitz boundary.
\end{enumerate}
\end{lemma}

%Especially, there exists $0 < s_{0} < 1$ such that for $0 < s \le
%s_{0}$ $G$ satisfies \eqref{target}. Since $G$ grows essentially
%more slowly near infinity than $G_{\frac{n}{s}}$, Lemma \ref{cpt}
%shows that for $0 < s \le s_{0}$
%\begin{align*}
%W^{s,G}(U) \rightarrow L^{G}(U)
%\end{align*}
%is compact for the bounded domain $U \subset \mr^{n}$ with Lipschitz boundary.

From the above lemma, we see the following compact embedding result.

\begin{lemma}\label{lem.embedding}  
Suppose an N-function $G$ satisfies \eqref{growth} and $0<s<1$. The embedding 
\begin{align*}
W^{s,G}(U) \rightarrow L^{G}(U)
\end{align*}
is compact for any bounded domain $U \subset \mr^{n}$ with Lipschitz boundary.
\end{lemma}
\begin{proof}
We note that $G$ grows essentially more slowly near infinity than $G_{\frac{n}{s}}$, and that $G$ satisfies \eqref{target} with $A=G$ if $0 < s < \frac{n}{q}$.  Therefore, Lemma \ref{cpt} directly implies the compact embedding from $W^{s,G}(U)$ to $L^{G}(U)$ when $s < \frac{n}{q}$.  Moreover, by taking advantage of the technique in \cite[Proposition 2.1]{DPV1}, one can see that the embedding $W^{s,G}(U)\to W^{\tilde s,G}(U)$ is continuous, i.e.,  $\|u\|_{W^{\tilde s,G}(U)}\le c \|u\|_{W^{s,G}(U)}$, for every $\tilde s\in (0,s)$. Therefore, the embedding $W^{s,G}(U) \rightarrow L^{G}(U)$ is also compact when $s \ge  \frac{n}{q}$.
\end{proof}

We next recall the following Poincar\'e type inequality.
\begin{lemma}(\cite[Corollary 6.2]{FS})
\label{poin} Let $U \subset \mr^{n}$ be a bounded open set with
Lipschitz boundary. Suppose $G$ is an N-function satisfying
\eqref{growth}. Then there exists a constant $c>0$ depending on
$n,s,p,q$ and $U$ such that
\begin{align*}
\int_{U}G(|u|)dx \le c\int_{\mr^{n}}\int_{\mr^{n}}G\left( \frac{|u(x)-u(y)|}{|x-y|^{s}} \right)\frac{dxdy}{|x-y|^{n}}
\end{align*}
for every $s \in (0,1)$ and $u \in W^{s,G}(U)$.
\end{lemma}

To prove the existence and uniqueness of weak solution to
\eqref{pde}, we consider the following the energy functional
\begin{align}\label{energy}
\mathcal I[v]:=\iint_{C_{\Omega}}G \left( \frac{|v(x)-v(y)|}{|x-y|^{s}} \right)K(x,y)\,dxdy,
\end{align}
where
\begin{align*}
v \in \mathcal{A}_{f}(\Omega):= \{v\in \mathbb W^{s,G}(\Omega):\ v=f \quad \mbox{in }\ \R^n\setminus\Omega\}.
\end{align*}
We say $u \in \mathcal{A}_{f}(\Omega)$ is a minimizer of $\mathcal I$ over $\mathcal{A}_{f}(\Omega)$
if $\mathcal I[u] \le \mathcal I[v]$ for all $v \in \mathcal{A}_{f}(\Omega)$.

\begin{theorem}
Let $\Omega\subset\R^n$ be a bounded domain, the operator
$\mathcal L$ and an N-function $G$ be given as in
Section~\ref{sec1}, and $f \in \mathbb{W}^{s,G}(\Omega)$. Then there
exists a unique minimizer $u$ of $\mathcal I$ over
$\mathcal{A}_{f}(\Omega)$. Moreover, a function $u \in
\mathcal{A}_{f}(\Omega)$ is the minimizer of $\mathcal I$ over
$\mathcal{A}_{f}(\Omega)$ if and only if it is the weak solution to
\begin{align}
\label{pde1}
\begin{cases}
\mathcal L u=0 &\mbox{in}\ \ \Omega, \\
u=f &\mbox{in}\ \ \mr^{n} \setminus \Omega.
\end{cases}
\end{align}
\end{theorem}

\begin{proof}
\textit{Step 1.} We first prove the existence and uniqueness of 
minimizer of $\mathcal I$. Since $f \in \mathcal{A}_{f}(\Omega)$,
$\mathcal{A}_{f}(\Omega)$ is nonempty. We now choose a minimizing
sequence $\{u_{m}\}_{m \ge 1}$ in $\mathcal{A}_{f}(\Omega)$ so that
$\mathcal I[u_{m}]$ is non-increasing and $\ds \lim_{m \rightarrow
\infty} \mathcal I[u_{m}] = \underset{w \in
\mathcal{A}_{f}(\Omega)}{\inf}\mathcal{I}[w]$. Set $v_{m}:=u_{m}-f$. Then
$\{v_{m}\}_{m \ge 1} \subset \mathbb{W}^{s,G}(\Omega)$ and $v_m=0$
in $\R^n\setminus \Omega$. Choose a ball $B=B_R(0)$ such that
$B\supset\Omega$. In order to use the compactness argument, we need
to show that $\|v_{m}\|_{W^{s,G}(B)}$ is bounded for $m$. Using
\eqref{lux}, \eqref{gar} and Lemma~\ref{poin}, and the fact that
$v_m=0$ in $\R^n\setminus \Omega$, we find
\begin{align*}
\|v_{m}\|_{W^{s,G}(B)} &\le \int_{B}G\left( |v_{m}| \right)dx +
\int_{B}\int_{B}G\left( \frac{|v_{m}(x)-v_{m}(y)|}{|x-y|^{s}} \right)\frac{dxdy}{|x-y|^{n}} + 2\\
&\le c\left[ \iint_{C_{\Omega}}G\left( \frac{|v_{m}(x)-v_{m}(y)|}{|x-y|^{s}} \right)K(x,y)\,dxdy + 2 \right]\\
&\le c(\mathcal I[u_{m}]+ \mathcal I[f]+2) \le c.
\end{align*}
Since $\mathcal I[u_{m}]$ is bounded, so is
$\|v_{m}\|_{W^{s,G}(B)}$. By the compactness result in
Lemma~\ref{lem.embedding}, there exist a subsequence $\{v_{m_{j}}\}_{j \ge
1}$ and $v \in W^{s,G}(B)$ such that
\begin{align*}%\label{v.conv}
\begin{cases}
v_{m_{j}} \rightharpoonup v &\mbox{weakly in}\  \ W^{s,G}(B), \\
v_{m_{j}} \rightarrow v &\mbox{in}\ \ L^{G}(B),\\
v_{m_{j}} \rightarrow v &\mbox{a.e. in}\ \ B,
\end{cases}
\quad \text{as} \quad j\to \infty.
\end{align*}
Now extend $v$ by zero outside $B$ and set $u=v+f$. Then we see that
$u\in \mathcal A_f(\Omega)$ and $\ds \lim_{j\to \infty}u_{m_{j}} =u$
a.e. in $\mr^{n}$. Therefore $v \in \mathbb{W}^{s,G}(\Omega)$ such
that $v=0$ in $\mr^{n} \setminus \Omega$ and so $v+f \in
\mathcal{A}_{f}(\Omega)$. Then Fatou's lemma implies
\begin{align*}
\mathcal I[u] \le \liminf_{j \rightarrow \infty} \mathcal I[u_{m_{j}}] =\underset{w \in \mathcal{A}_{f}}{\inf}\mathcal{I}[w].
\end{align*}
The uniqueness directly follows from the convexity of $G$. Indeed,
to prove this, we first suppose that $u,v \in
\mathcal{A}_{f}(\Omega)$ are two different minimizers of $\mathcal
I$. Then $\mathcal{I}[u] = \mathcal{I}[v]$. Since $G$ is strictly convex, we have
\begin{align*}
\mathcal I[u] \le\mathcal I\left[ \frac{u+v}{2} \right] <
\frac{\mathcal I[u]+\mathcal I[v]}{2} =\mathcal I[u],
\end{align*}
which is a contradiction.

\textit{Step 2.} We next show the equivalence between the minimizer
of \eqref{energy} and a weak solution to \eqref{pde1}. Suppose $u$
is the minimizer of \eqref{energy}. Then for any $\eta\in \mathbb
W^{s,G}(\Omega)$ with $\eta=0$ in $\R^n\setminus \Omega$, $\mathcal
I[u+\tau \eta]$ has a critical point at $\tau=0$. Thus
\begin{align*}
0 &= \frac{d}{d\tau}\mathcal{I}[u+\tau \eta]\Big |_{\tau=0}\\
&= \iint_{C_{\Omega}}\frac{d}{d\tau}G\left( \frac{|u(x)-u(y)+\tau(\eta(x)-\eta(y))|}{|x-y|^{s}} \right)\Big |_{\tau=0}K(x,y)\, dxdy\\
&= \iint_{C_{\Omega}}g\left( \frac{|u(x)-u(y)|}{|x-y|^{s}} \right)\frac{u(x)-u(y)}{|u(x)-u(y)|}(\eta(x)-\eta(y))K(x,y)\,\frac{dxdy}{|x-y|^{s}}.
\end{align*}
Therefore $u$ is a weak solution to \eqref{pde1}.

On the other hand, suppose $u$ is a weak solution to \eqref{pde1}.
Then for any $v \in \mathcal{A}_{f}(\Omega)$, we see that $u-v \in
\mathbb{W}^{s,G}(\mr^{n})$ and that $u-v=0$ in $\mr^{n} \setminus
\Omega$. We then test $u-v$ in the weak formulation of \eqref{pde1}
to discover
\begin{align}\label{eu.test}
\begin{split}
0&= \iint_{C_{\Omega}}g\left(\frac{|u(x)-u(y)|}{|x-y|^s}\right)\frac{u(x)-u(y)}{|u(x)-u(y)|}(\eta(x)-\eta(y))
K(x,y)\frac{dxdy}{|x-y|^{s}}\\
&= \iint_{C_{\Omega}}g\left(\frac{|u(x)-u(y)|}{|x-y|^s}\right)|u(x)-u(y)|K(x,y)\frac{dxdy}{|x-y|^{s}}\\
&\qquad -\iint_{C_{\Omega}}g\left(\frac{|u(x)-u(y)|}{|x-y|^s}\right)
\frac{u(x)-u(y)}{|u(x)-u(y)|}(v(x)-v(y))K(x,y)\frac{dxdy}{|x-y|^{s}}.
\end{split}
\end{align}
Let us look at the integrand of the second term with respect to the
measure $K(x,y)\,dxdy$ on the right-hand side. From \eqref{Young0} and
\eqref{Young1}, we see
\begin{align}
\label{eu.young}
\begin{split}
&g\left(\frac{|u(x)-u(y)|}{|x-y|^{s}}\right)\frac{u(x)-u(y)}{|u(x)-u(y)|}\frac{v(x)-v(y)}{|x-y|^{s}}\le
g\left(\frac{|u(x)-u(y)|}{|x-y|^{s}}\right)\frac{|v(x)-v(y)|}{|x-y|^{s}}\\
&\le G^{*}\left( g\left( \frac{|u(x)-u(y)|}{|x-y|^s} \right) \right) + G\left( \frac{|v(x)-v(y)|}{|x-y|^{s}} \right)\\
&= \frac{|u(x)-u(y)|}{|x-y|^s}g\left( \frac{|u(x)-u(y)|}{|x-y|^{s}}
\right) - G\left( \frac{|u(x)-u(y)|}{|x-y|^{s}} \right) + G\left(
\frac{|v(x)-v(y)|}{|x-y|^{s}} \right).
\end{split}
\end{align}
We combine \eqref{eu.young} and \eqref{eu.test} to conclude that
$\mathcal I[v] \ge \mathcal I[u]$. Therefore $u$ is the minimizer of $\mathcal I$.
\end{proof}


\begin{bibdiv}
\begin{biblist}

\bib{ACPS}{article}{
      author={Alberico, A.},
      author={Cianchi, A.},
      author={Pick, L.},
      author={Slav\'{\i}kov\'{a}, L.},
       title={On fractional {O}rlicz-{S}obolev spaces},
        date={2021},
     journal={Anal. Math. Phys.},
      volume={11},
      number={2},
       pages={Paper No. 84, 21},
}

\bib{BOS}{article}{
  title={H\"older regularity for weak solutions to nonlocal double phase problems},
  author={Byun, S.},
   author={Ok, J.}
   author={Song, K.},
  journal={arXiv:2108.09623},
  year={2021}
}


\bib{CCV1}{article}{
  title={Regularity theory for parabolic nonlinear integral operators},
  author={Caffarelli, L.},
  author={Chan, C. H.},
author={Vasseur, A.},
        date={2011},
     journal={J. Amer. Math. Soc.},
      volume={24},
      number={3},
       pages={849\ndash 869},
}


\bib{CK1}{article}{
      author={Chaker, J.},
      author={Kim, M.},
      title={Local regularity for nonlocal equations with variable exponents}, 
      journal={arXiv.2107.06043},
      year={2021},
}



\bib{CKW1}{article}{
      author={Chaker, J.},
      author={Kim, M.},
      author={Weidner, M.},
       title={Regularity for nonlocal problems with non-standard growth},
     journal={arXiv:2111.09182},
  year={2021}
}

\bib{Coz}{article}{
   author={Cozzi, M.},
   title={Regularity results and Harnack inequalities for minimizers and
   solutions of nonlocal problems: a unified approach via fractional De
   Giorgi classes},
   journal={J. Funct. Anal.},
   volume={272},
   date={2017},
   number={11},
   pages={4762--4837},
}

\bib{CS}{article}{
   author={Caffarelli, L.},
   author={Silvestre, L.},
   title={Regularity theory for fully nonlinear integro-differential
   equations},
   journal={Comm. Pure Appl. Math.},
   volume={62},
   date={2009},
   number={5},
   pages={597\ndash 638},
}

\bib{DE}{article}{
      author={Diening, L.},
      author={Ettwein, F.},
       title={Fractional estimates for non-differentiable elliptic systems with
  general growth},
        date={2008},
     journal={Forum Math.},
      volume={20},
      number={3},
       pages={523\ndash 556},
}

\bib{DKP1}{article}{
   author={Di Castro, A.},
   author={Kuusi, T.},
   author={Palatucci, G.},
   title={Nonlocal Harnack inequalities},
   journal={J. Funct. Anal.},
   volume={267},
   date={2014},
   number={6},
   pages={1807--1836},
}

\bib{DKP2}{article}{
      author={Di Castro, A.},
      author={Kuusi, T.},
      author={Palatucci, G.},
       title={Local behavior of fractional $p$-minimizers},
        date={2016},
     journal={Ann. Inst. H. Poincar\'{e} Anal. Non Lin\'{e}aire},
      volume={33},
      number={5},
       pages={1279\ndash 1299},
}

\bib{DLO}{article}{
      author={Diening, L.},
      author={Lee, M.},
      author={Ok, J.},
      title={Parabolic weighted sobolev-poincar\'e type inequalities},
      journal={arXiv:2108.00656},
      year={2021},
}



\bib{DSV1}{article}{
      author={Diening, L.},
      author={Stroffolini, B.},
      author={Verde, A.},
      title={Everywhere regularity of functionals with $\varphi$-growth},
      journal={Manuscripta Math.},
   volume={129},
   date={2009},
   number={4},
   pages={449--481},
}




\bib{DPV1}{article}{
      author={Di~Nezza, E.},
      author={Palatucci, G.},
      author={Valdinoci, E.},
      title={Hitchhiker's guide to the fractional Sobolev spaces},
      journal={Bull. Sci. Math.},
   volume={136},
   date={2012},
   number={5},
   pages={521--573},
}





\bib{DP}{article}{
   author={De Filippis, C.},
   author={Palatucci, G.},
   title={H\"{o}lder regularity for nonlocal double phase equations},
   journal={J. Differential Equations},
   volume={267},
   date={2019},
   number={1},
   pages={547--586},
}

\bib{DZZ}{article}{
   author={Ding, M.},
   author={Zhang, C.},
   author={Zhou, S.},
   title={Local boundedness and H\"{o}lder continuity for the parabolic
   fractional $p$-Laplace equations},
   journal={Calc. Var. Partial Differential Equations},
   volume={60},
   date={2021},
   number={1},
   pages={Paper No. 38, 45},
}


\bib{FS}{article}{
      author={Fern\'{a}ndez Bonder, J.},
      author={Salort, A. M.},
       title={Fractional order {O}rlicz-{S}obolev spaces},
        date={2019},
     journal={J. Funct. Anal.},
      volume={277},
      number={2},
       pages={333\ndash 367},
}

\bib{FSV1}{article}{
  title={Interior and up to the boundary regularity for the fractional $g$-Laplacian: the convex case},
  author={Fern\'{a}ndez Bonder, J.},
  author={Salort, A.},
 author={Vivas, H.},
  journal={arXiv:2008.05543},
  year={2020}
}


\bib{FSV2}{article}{
  title={Global Hölder regularity for eigenfunctions of the fractional g-Laplacian},
  author={Fern\'{a}ndez Bonder, J.},
  author={Salort, A.},
 author={Vivas, H.},
  journal={arXiv:2112.00830},
  year={2021}
}



\bib{FZ}{article}{
      author={Fang, Y.},
      author={Zhang, C.},
      title={On weak and viscosity solutions of nonlocal double phase equations},
      journal={arXiv:2106.04412},
      year={2021},
}

\bib{G}{book}{
      author={Giusti, E.},
       title={Direct methods in the calculus of variations},
   publisher={World Scientific Publishing Co., Inc., River Edge, NJ},
        date={2003},
}

\bib{GK}{article}{
      author={Garain, P.},
      author={Kinnunen, J.},
      title={On the regularity theory for mixed local and nonlocal quasilinear elliptic equations},
      journal={arXiv2102.13365},
      year={2021},
}


\bib{GKS}{article}{
      author={Giacomoni, J.},
      author={Kumar, D.},
      author={Sreenadh, K.},
      title={Interior and boundary regularity results for strongly nonhomogeneous $p,q$-fractional problems},
      journal={arXiv:2102.06080},
      year={2021},
}

\bib{HH}{book}{
      author={Harjulehto, P.},
      author={H\"{a}st\"{o}, P.},
       title={Orlicz spaces and generalized {O}rlicz spaces},
      series={Lecture Notes in Mathematics},
   publisher={Springer, Cham},
        date={2019},
      volume={2236},
}





\bib{HHL1}{article}{
 author={Harjulehto, P.},
   author={H\"{a}st\"{o}, P.},
   author={Lee, M.},
   title={H\"older continuity of quasiminimizers and $\omega$-minimizers of functionals with generalized Orlicz growth},
   journal={Ann. Sc. Norm. Super. Pisa Cl. Sci.},
   volume={XXII},
   date={2021},
   number={2},
   pages={549--582},
}


\bib{HO}{article}{
   author={H\"{a}st\"{o}, P.},
   author={Ok, J.},
   title={Maximal regularity for local minimizers of non-autonomous functionals},
   journal={J. Eur. Math. Soc., DOI: 10.4171/JEMS/1118},
   }



%\bib{JLW}{article}{
%   author={Jia, H.},
%   author={Li, D.},
%   author={Wang, L.},
%   title={Regularity theory in Orlicz spaces for elliptic equations in
%   Reifenberg domains},
%   journal={J. Math. Anal. Appl.},
%   volume={334},
%   date={2007},
%   number={2},
%   pages={804--817},
%}

\bib{Kas1}{article}{
   author={Kassmann, M.},
   title={The theory of De Giorgi for non-local operators},
   language={English, with English and French summaries},
   journal={C. R. Math. Acad. Sci. Paris},
   volume={345},
   date={2007},
   number={11},
   pages={621--624},
}

\bib{Kas2}{article}{
   author={Kassmann, M.},
   title={A priori estimates for integro-differential operators with
   measurable kernels},
   journal={Calc. Var. Partial Differential Equations},
   volume={34},
   date={2009},
   number={1},
   pages={1--21},
}

\bib{KKL}{article}{
   author={Korvenp\"{a}\"{a}, J.},
   author={Kuusi, T.},
   author={Lindgren, E.},
   title={Equivalence of solutions to fractional $p$-Laplace type equations},
   language={English, with English and French summaries},
   journal={J. Math. Pures Appl. (9)},
   volume={132},
   date={2019},
   pages={1\ndash 26},
}

\bib{KKP1}{article}{
   author={Korvenp\"{a}\"{a}, J.},
   author={Kuusi, T.},
   author={Palatucci, G.},
   title={Fractional superharmonic functions and the Perron method for
   nonlinear integro-differential equations},
   journal={Math. Ann.},
   volume={369},
   date={2017},
   number={3-4},
   pages={1443--1489},
}

\bib{KKP2}{article}{
   author={Korvenp\"{a}\"{a}, J.},
   author={Kuusi, T.},
   author={Palatucci, G.},
   title={The obstacle problem for nonlinear integro-differential operators},
   journal={Calc. Var. Partial Differential Equations},
   volume={55},
   date={2016},
   number={3},
   pages={Art. 63, 29},
}

%\bib{KL}{article}{
%   author={Karppinen, A.},
%   author={Lee, M.},      
%   title={H\"older continuity of the minimizer of an obstacle problem with generalized Orlicz growth}, 
%   journal={arXiv:2006.08244},
%   year={2021},
%}

\bib{KMS1}{article}{
   author={Kuusi, T.},
   author={Mingione, G.},
   author={Sire, Y.},
   title={Nonlocal equations with measure data},
   journal={Comm. Math. Phys.},
   volume={337},
   date={2015},
   number={3},
   pages={1317\ndash 1368},
}

\bib{KMS2}{article}{
   author={Kuusi, T.},
   author={Mingione, G.},
   author={Sire, Y.},
   title={Nonlocal self-improving properties},
   journal={Anal. PDE},
   volume={8},
   date={2015},
   number={1},
   pages={57\ndash 114},
}

\bib{Lie1}{article}{
      author={Lieberman, G. M.},
       title={The natural generalization of the natural conditions of Ladyzhenskaya and Ural’tseva for elliptic equations},
        date={1991},
     journal={Comm. Partial Differential Equations},
      volume={16},
      number={2-3},
       pages={311\ndash 361},
}

\bib{Lin}{article}{
   author={Lindgren, E.},
   title={H\"{o}lder estimates for viscosity solutions of equations of
   fractional $p$-Laplace type},
   journal={NoDEA Nonlinear Differential Equations Appl.},
   volume={23},
   date={2016},
   number={5},
   pages={Art. 55, 18},
}

\bib{MSY}{article}{
   author={Mengesha, T.},
   author={Schikorra, A.},
   author={Yeepo, S.},
   title={Calderon-Zygmund type estimates for nonlocal PDE with H\"{o}lder
   continuous kernel},
   journal={Adv. Math.},
   volume={383},
   date={2021},
   pages={Paper No. 107692, 64},
}

\bib{MR}{article}{
      author={Mih\u{a}ilescu, M.},
      author={R\u{a}dulescu, V.},
       title={Neumann problems associated to nonhomogeneous differential
  operators in {O}rlicz-{S}obolev spaces},
        date={2008},
     journal={Ann. Inst. Fourier (Grenoble)},
      volume={58},
      number={6},
       pages={2087\ndash 2111},
}

\bib{Now1}{article}{
   author={Nowak, S.},
   title={Higher H\"{o}lder regularity for nonlocal equations with irregular
   kernel},
   journal={Calc. Var. Partial Differential Equations},
   volume={60},
   date={2021},
   number={1},
   pages={Paper No. 24, 37},
}

\bib{Now2}{article}{
      author={Nowak, S.},
      title={Regularity theory for nonlocal equations with VMO coefficients},
      journal={arXiv:2101.11690},
      year={2021},
}

\bib{Ok18}{article}{
      author={Ok, J.},
       title={Partial Hölder regularity for elliptic systems with non-standard growth},
        date={2018},
     journal={J. Funct. Anal.},
      volume={274},
      number={3},
       pages={723\ndash 768},
}

\bib{Ok1}{article}{
      author={Ok, J.},
       title={Local {H}\"older regularity for nonlocal equations with
  variable powers},
     journal={arXiv:2107.06611},
     year={2021} 
}


\bib{Sil}{article}{
   author={Silvestre, L.},
   title={H\"{o}lder estimates for solutions of integro-differential equations
   like the fractional Laplace},
   journal={Indiana Univ. Math. J.},
   volume={55},
   date={2006},
   number={3},
   pages={1155--1174},
}


\end{biblist}
\end{bibdiv}
\end{document}